\documentclass[a4paper,12]{elsarticle}

\usepackage[utf8]{inputenc}
\usepackage{latexsym,amsfonts,amssymb,amsthm,amsmath,amscd,mathrsfs}
\usepackage{xargs}   % Use more than one optional parameter in a new commands
\usepackage{xspace}
\usepackage[shortlabels]{enumitem}
\usepackage{graphicx}
\usepackage{float}
\usepackage{multirow}

\usepackage{caption}
\usepackage{subcaption}

\usepackage{hyperref}
\usepackage{graphicx,tikz,times,pgffor}
\usepackage{pgfplots}
\usetikzlibrary{calc,through,backgrounds,patterns}
\usepackage{tabulary}
\usepackage{tabularx}
\pgfplotsset{compat=newest} 
\pgfplotsset{plot coordinates/math parser=false} 
\newlength\figureheight 
\newlength\figurewidth 

\usepackage{xpatch}
\usepackage[ruled,vlined]{algorithm2e}
\newtheorem{theorem}{Theorem}
\newtheorem{definition}[theorem]{Definition}

\newtheorem{lemma}[theorem]{Lemma}
\newtheorem{property}[theorem]{Property}
\newtheorem{remark}[theorem]{Remark}

\newcommand{\x} {\mathrm{x}}

\newcommand{\setX}{\mathcal{X}}

\newcommand{\setY}{\mathcal{Y}}

\newcommand{\setB}{\mathcal{B}}

\newcommand{\Rn}{\mathcal{R}}

\newcommand{\I}{\mathbb{I}}
\newcommand{\U}{\mathbb{U}}

\newcommand{\R}{\mathbb{R}}

\newcommand {\bmat} {\left[\begin{array} }
\newcommand {\emat} {\end{array}\right]}

\DeclareMathOperator{\inti}{int}

\DeclareMathOperator{\tot}{tot}

\newenvironment{pf*}[1][Proof:]{\noindent \textbf{#1} }{}

\usepackage[shadow,colorinlistoftodos,prependcaption,spanish,textsize=footnotesize]{todonotes}
\newcommandx{\mara}[2][1=]{\todo[linecolor=red,backgroundcolor=red!25,bordercolor=red,#1]{#2}}
\newcommandx{\otro}[2][1=]{\todo[linecolor=blue,backgroundcolor=blue!25,bordercolor=blue,#1]{#2}}
\newcommandx{\ale}[2][1=]{\todo[linecolor=OliveGreen,backgroundcolor=OliveGreen!25,bordercolor=OliveGreen,#1]{#2}}
\newcommandx{\pablo}[2][1=]{\todo[linecolor=Plum,backgroundcolor=Plum!25,bordercolor=Plum,#1]{#2}}
\newcommandx{\marce}[2][1=]{\todo[linecolor=lime,backgroundcolor=lime!25,bordercolor=lime,#1]{#2}}
\definecolor{ao_english}{rgb}{0.0, 0.5, 0.0}
\definecolor{cagreen}{rgb}{0.12, 0.3, 0.17}
\definecolor{cgreen}{rgb}{0.0, 0.42, 0.24}
\definecolor{calpolypomonagreen}{rgb}{0.12, 0.3, 0.17}
\definecolor{armygreen}{rgb}{0.29, 0.33, 0.13}
\definecolor{darkpastelred}{rgb}{0.76, 0.23, 0.13}
\definecolor{burntumber}{rgb}{0.54, 0.2, 0.14} 
\definecolor{brown(web)}{rgb}{0.65, 0.16, 0.16}
\definecolor{bostonuniversityred}{rgb}{0.8, 0.0, 0.0}
\definecolor{cornellred}{rgb}{0.7, 0.11, 0.11}
\definecolor{blue(pigment)}{rgb}{0.2, 0.2, 0.6}
\definecolor{ceruleanblue}{rgb}{0.16, 0.32, 0.75}
\definecolor{darkblue}{rgb}{0.0, 0.0, 0.55}

\numberwithin{equation}{section}
\numberwithin{theorem}{section}

\allowdisplaybreaks
\begin{document}
\begin{frontmatter}

\title{Optimal control strategies to tailor antivirals for acute infectious diseases in the host}

\author[First]{Mara Perez} 
%\author[First]{Juan Sereno} 
\author[First]{Pablo Abuin}
%\author[First]{Alejandro Anderson} 
%\author[First]{Agustina D'Jorge}
\author[Second]{Marcelo Actis}
\author[Third]{Antonio Ferramosca} 
\author[Forth]{Esteban A. Hernandez-Vargas}
\author[First]{Alejandro H. Gonzalez} 

\address[First]{Institute of Technological Development for the Chemical Industry (INTEC), CONICET-UNL, Santa Fe, Argentina.}
\address[Second]{Facultad de Ingeniería Química (FIQ), Universidad Nacional del Litoral (UNL) and Consejo Nacional de Investigaciones científicas y técnicas (CONICET), Santa Fe, Argentina.}
\address[Third]{Department of Management, Information and Production Engineering, University of Bergamo Via Marconi 5, Dalmine (BG) 24044, Italy.}
\address[Forth]{Instituto de Matematicas, Unidad Juriquilla, UNAM, Mexico}
%\address[Thirda]{Instituto de Matemáticas, Universidad Nacional Autonoma de Mexico, Boulevard Juriquilla 3001, Querétaro, Qro., 76230, México}
%\footnote{Corresponding author: alejgon@santafe-conicet.gov.ar, vargas@fias.uni-frankfurt.de}

%%%%%%%%%%%%%%%%%%%%%%%%%%%%%%%%%%%%%%%%%%%%%%%%%%%%%%%%%%%%%%%%%%%%%%%%%%%%%%%%
\begin{abstract}
	 Several mathematical models in SARS-CoV-2 have shown how target-cell model can help to understand the spread of the virus in the host and how potential candidates of antiviral treatments can help to control the virus. Concepts as equilibrium and stability show to be crucial to qualitative determine the best alternatives to schedule drugs, according to effectivity in inhibiting the virus infection and replication rates. 
	 Important biological events such as rebounds of the infections (when antivirals are incorrectly interrupted) can also be explained by means of a dynamic study of the target-cell model. In this work a full characterization of the dynamical behavior of the target-cell models under control actions is made and, based on this characterization, the optimal fixed-dose antiviral schedule that produces the smallest amount of dead cells (without viral load rebounds) is computed. Several simulation results - performed by considering real patient data - show the potential benefits of both, the model characterization and the control strategy.
\end{abstract}
\begin{keyword}
	In-host acute infection model, Equilibrium sets characterization, Stability analysis, Model predictive control.
\end{keyword}
\end{frontmatter}

%%%%%%%%%%%%%%%%%%%%%%%%%%%%%%%%%%%%%%%%%%%%%%%%%%%
\section{Introduction}\label{sec:intro}
%%%%%%%%%%%%%%%%%%%%%%%%%%%%%%%%%%%%%%%%%%%%%%%%%%%
%
Mathematical models of with-in infections can be used to characterize pathogen dynamics, optimize drug delivery, uncover biological parameters (including pathogen and infected cell half-lives), design clinical trials, among others. They have been employed to study chronic (\textit{i.e.:} HIV\cite{perelson1993dynamics,legrand2003vivo,perelson2013modeling}, hepatitis B\cite{ciupe2007modeling,herrmann2000hepatitis}, hepatitis C\cite{neumann1998hepatitis,canini2014viral}) and acute (\textit{i.e.:} influenza \cite{baccam2006kinetics,smith2011influenza,hernandez2019passivity}, dengue\cite{nikin2015role,nikin2018modelling}, Ebola\cite{Nguyen15}) infections.  
Currently, they are based on ordinary differential equations (ODE), which allows to analyze these systems employing mathematical and computational tools. This way, in-host basic reproduction numbers ($\Rn$), stability analysis of equilibrium states, analytical/numerical solutions, can be computed \cite{van2017reproduction,murase2005stability,smith2003virus,abuin2020dynamical}. 
Most of them are based on the target-cell limited model to represent chronic/acute infections according to the infection resolution respect to the target cell production and natural death rates \cite{ciupe2017host}. This way, the equilibrium states differ from isolated equilibrium points (\textit{i.e.:} disease free and infected equilibria) for the former to a continuous of equilibrium points (\textit{i.e.:} disease free equilibrium set) for the latter ones. Note that for acute infections, the only feasible equilibria is the disease free, since the pathogen particles at the end of infection will be cleared independently of the in-host reproduction number \cite{ciupe2017host,baccam2006kinetics,cao2017mechanisms}. The existence of healthy equilibrium set implies that stability analysis can be performed considering equilibrium sets as a generalization of equilibrium points, which gives an environment to employ set-theoretic methods \cite{rawlings2017model,blanchini2008set}, widely used in the design of set-based controllers, although not fully employed for modeling characterization and control of acute infections. Some preliminary results, which will be discuss later in this chapter, can be found in \cite{abuin2020char}.

The control of infection can be modelled considering immune response mechanisms, where the infection is self-controlled by a combination of a non-specific and specific reactions \cite{eftimie2016mathematical,hernandez2019modeling,dobrovolny2013assessing}, or by drug therapies. The inclusion of pharmacokinetic (PK) and pharamacodynamic (PD) models of drug therapies allows the inclusion of therapeutic effects on the pathogen evolution \cite{canini2014viral,ciupe2017host}. Therefore, the models parameters can be changed exogenously by dose frequency and quantity, naturally limited by the inhibitory potential of the drug (expressed in terms of EC50,
or drug concentration for inhibiting $50\%$ of antigen particles) and its cytotoxic effect (expressed in terms of
IC50, or drug concentration which causes death to $50\%$ of susceptible cells) \cite{vergnaud2005assessing}.  Moreover, since drugs are normally administrated by pills or intravenous injections, instantaneous jumps are observed in the concentration of the drug in some tissues. This is mathematically conceptualized as a discontinuity of the first kind and gives rise to the so-called impulsive control systems \cite{rivadeneira2017control}. This model representation has been used for optimal control and state-feedback control with constraints for infectious diseases, such as: influenza \cite{hernandez2019passivity,HernandezMejia17} and HIV \cite{rivadeneira2012impulsive,rivadeneira2017control}. Even though optimal dosage can be computed for chronic and acute models, the unstable healthy equilibrium of the former (under certain circumstances; for details, see \cite{ciupe2017host,murase2005stability}) and the availability of target cells above a critical level for the latter (as it is discuss later),  involve the duration of drug therapy, with the presence of viral rebounds when therapy is disrupted. This effect has been noticed for chronic \cite{dahari2007modeling} and acute  \cite{dobrovolny2011neuraminidase} infections. Taking
into account this scenario, in this work, we formalize the existence of an optimal single interval drug delivery such that viral rebounds are avoided. Even though, the presented analysis is valid for the target-cell limited model for acute infections,  taking into account the current worldwide contextual situation (COVID-19 pandemic), we prove our results using an identified model of infected patients with SARS-CoV-2 virus \cite{wolfel2020virological,abuin2020char,vargas2020host}. 

After the introduction given in Section~\ref{sec:intro} the article is organized as follows. 
Section~\ref{sec:infectmod} presents the general "in the host" models used to represent infectious diseases. Section~\ref{sec:pkpd} studies the way the antivirals affect the dynamic of the model, emphasizing the fact that the stability analysis made in Section~\ref{sec:equil_stab} remains unmodified and, so, any control strategy must be designed accounting for these details. In Section~\ref{sec:control} control design able to exploit the stability model characterization is introduced, and its benefits are shown by simulating several cases, in Section~\ref{sec:simres}. Finally, conclusions are given in Section~\ref{sec:conc}. 

%%%%%%%%%%%%%%%%%%%%%%%%%%%%%%%%%%%%%%%%%%%%%%%%%%%
\subsection{Notation}\label{sec:notation}
%%%%%%%%%%%%%%%%%%%%%%%%%%%%%%%%%%%%%%%%%%%%%%%%%%%

First let us introduce some basic notation. We consider $\R^n$ as $n$-dimensional Euclidean space equipped with the euclidean distance between two points defined by $d(x,y):=\|x-y\|=[(x-y)'(x-y)]^{1/2}$. 
The euclidean distance from a point $x$ to a set $\setY$ is given by $d(x,\setY):= \|x\|_\setY=\inf\{y\in\setY: \|x-y\|\}$. 

With $\setX$ we will denote the constraint set of $\R^3$, given by
\[\setX:=\R^3_{\geq0}=\{(x_1,x_2,x_3)\in\R^3: x_1\geq0, x_2\geq0 \text{ and } x_3\geq0\}\]
We will consider $\setX$ endowed with the \emph{inherit topology} of $\R^3$, i.e. the open sets are intersections of open set of $\R^3$ with $\setX$. Thus, a open ball in $\setX$ with center in $x$ and radius $\varepsilon>0$ is given by {$\setB_\varepsilon(x):=\{y \in \setX: \|x- y\|<\varepsilon\}$} and an $\varepsilon$-neighborhood of set $\setY\subset\setX$ is given by $\setB_{\varepsilon}(\setY):=\{x\in\setX: \|x\|_\setY<\varepsilon\}$.
Let $x \in\setY$, we say that $x$ is an interior point of $\setY$ if the there exist $\varepsilon > 0$ such
that $\setB_\varepsilon(x) \subseteq \setY$. The interior of $\setY$ is the set of all interior points of $\setY$ and it is denoted by $\inti(\setY)$.

%%%%%%%%%%%%%%%%%%%%%%%%%%%%%%%%%%%%%%%%%%%%%%%%%%%%%%%%%%%%%%%%%%%%%%%%%%%%%%%%%%%%%%%%%%%%%%%%
\section{Review of the UIV target-cell-limited model}\label{sec:infectmod}
%%%%%%%%%%%%%%%%%%%%%%%%%%%%%%%%%%%%%%%%%%%%%%%%%%%%%%%%%%%%%%%%%%%%%%%%%%%%%%%%%%%%%%%%%%%%%%%%

Mathematical models of in-host virus dynamic have shown to be useful to understand of the interactions that govern infections and, more important, to allows external intervention to moderate their effects \cite{hernandez2019modeling}. According to recent research in the area \cite{ciupe2017host,abuin2020char}, the following ordinary differential equations (ODEs) are used in this work to describe the interaction between uninfected target or susceptible cells $U$ [cell/mm$^3$], 
infected cells $I$ [cell/mm$^3$], and virus $V$ [copies/mL]:
\begin{subequations}\label{eq:SysOrig}
	\begin{align}
	&\dot{U}(t)   =  -\beta U(t) V(t),~~~~~~~ U(0) = U_0,\\
	&\dot{I}(t) =  \beta U(t) V(t) - \delta I(t),~~~~~~~I(0) = I_0,\\
	&\dot{V}(t)   = pI(t) - c V(t),~~~~~~~ V(0) = V_0,
	\end{align}
\end{subequations}
where $\beta$ [mL.day$^{-1}/$copies] is the infection rate of healthy $U$ cells by external virus $V$, $\delta$ [day$^{-1}$] is the death rates of $I$, $p$ [(copies.mm$^3/$cell.mL).day$^{-1}$] is the production rate of free virus from infected cells $I$, and $c$ [day$^{-1}$] is degradation (or clearance) rate of virus $V$ by the immune system.

System \eqref{eq:SysOrig} is positive, which means that $U(t) \geq 0$, $I(t) \geq 0$ and $V(t) \geq 0$, for all $t\geq0$. We denote $x(t):=(U(t),I(t),V(t))$ the state vector, and	$\setX = \R^3_{\geq0}$ the state constraints set. 

The initial conditions of (\ref{eq:SysOrig}), which represent a healthy steady state before the infection, are assumed to be $V(t) = 0$, $I(t) = 0$, and $U(t)= U_0>0$, for $t<0$. Then, at time $t=0$, a small quantity of virions enters the host body and, so, a discontinuity occurs in $V(t)$. Indeed, $V(t)$ jumps from $0$ to a small positive value $V_0$ at $t=0$ (formally, $V(t)$ has a discontinuity of the first kind at $t_0$, \textit{i.e.}, $\lim_{t\to0^-} V(t)=0$ while $\lim_{t\to0^+} V(t)=V_0>0$). \\

Although the solution of \eqref{eq:SysOrig} for $t \geq t_0$, being $t_0$ an arbitrary time, is unknown, we know that it depends on the \textbf{basic reproduction number}\footnote{The reproduction number is usually defined as $\Rn=U(t_0)\frac{\beta p}{c \delta}>0$, for UIV-type models. However, for the sake of convenience, we remove the initial value of $U$ in our definition.} $\Rn=\frac{\beta p}{c \delta}$ and the initial conditions $(U(t_0),I(t_0),V(t_0))\in \setX$. Since $U(t) \geq 0$, $V(t) \geq 0$, for all $t \geq t_0$, $U(t)$ is a non increasing function of $t$ (by~\ref{eq:SysOrig}.a). From \cite{abuin2020char} and \cite{abuin2020dynamical}, if $c>>\delta$ (as it is always the case\footnote{If $c>>\delta$. system \eqref{eq:SysOrig} can be approximated by $\dot{U}(t) \approx -\beta U(t) V(t)$, $\dot{V}(t) \approx  (\frac{\beta p}{c} U(t) - \delta) V(t)$, $I(t)=\frac{c}{p} V(t)$. Then, since $U(t_0)>0$, conditions for $V$ to increase or decrease at $t_0$, are given by $U(t_0) \Rn > 1$ and $U(t_0) \Rn < 1$, respectively.}) it is known that for $U(t_0)\Rn \leq 1$, then $V(t)$ is a non increasing function of $t$, for all $t \geq t_0$, and goes asymptotically to zero for $t\rightarrow \infty$. On the other hands, if $U(t_0)\Rn>1$, $V(t)$ reaches a maximum $\hat V$ and then goes asymptotically to zero, for $t\rightarrow \infty$. In this latter case it is said that the virus \textbf{spreads in the host}, since there is at least one time instant for which $\dot{V}>0$ \cite{abuin2020char}.
The so called \textbf{critical value of} $U$, $U^*$, is defined as
\begin{eqnarray}\label{eq:Ustar}
	U^*:= 1/\Rn,
\end{eqnarray}
where $\Rn$ is assumed to remain constant for all $t\geq t_0$.
The critical value $U^*$ can be seen as the counterpart of the "herd immunity" in the epidemiological SIR-type models: i.e., $U(t)$ reaches $U^*$ approximately at the same time as $V(t)$ and $I(t)$ reach their peaks or, in other words, $V(t)$ and $I(t)$ cannot increase anymore once $U(t)$ is below $U^*$. This way, conditions $U(t_0)\Rn>1$ and $U(t_0)\Rn \leq 1$ that determines if $V(t)$ increases or decreases for $t\geq t_0$ can be rewritten as $U(t_0)>U^*$ and $U(t_0) \leq U^*$, respectively. In what follows, we assume that $U(0)>U^*$ (or $U(0)\Rn>1$), which corresponds to the case of the outbreak of the infection (i.e., the virus does spread in the host), at time $t=0$.\\

Let us now define $U_\infty:= \lim_{t\rightarrow \infty} U(t)$, $V_\infty:= \lim_{t \rightarrow \infty} V(t)$ and $I_\infty:= \lim_{t\rightarrow \infty} I(t)$, which are values that depend on $\Rn$ and the initial conditions $U(t_0),V(t_0)$, $I(t_0)$. According to \cite{abuin2020char}, $V_\infty=I_\infty=0$, while $U_\infty$ is a value in $(0,U(t_0))$, which will be characterized in the next section.

%%%%%%%%%%%%%%%%%%%%%%%%%%%%%%%%%%%%%%%%%%%
\section{Equilibria characterization and stability} \label{sec:equil_stab}
%%%%%%%%%%%%%%%%%%%%%%%%%%%%%%%%%%%%%%%%%%%

To find the equilibrium set of model \eqref{eq:SysOrig}, with initial conditions $(U(t_0),V(t_0),I(t_0)) \in \setX$ at an arbitrary time $t_0 \geq 0$, $\dot{U}(t)$, $\dot{I}(t)$ and $\dot{V}(t)$ need to be equaled to zero, in \eqref{eq:SysOrig}. According to \cite{abuin2020char,abuin2020dynamical} there is only one equilibrium set in $\setX$, which is a healthy one, and it is defined by 
\begin{eqnarray}\label{eq:equilset}
	\setX_s := \{(U,I,V)\in \setX :~I=0,~V=0\},
\end{eqnarray}
%

%%%%%%%%%%%%%%%%%%%%%%%%%%%%%%%%%%%%%%%%%%%%%%%%%%%%%%%%%%%%%%%%%%%%%%%%%%%%
To examine the stability of the equilibrium points in $\setX_s$, a first attempt consists in linearizing system \eqref{eq:SysOrig} at some state $x_s:=(U_s,I_s,V_s) \in \setX_s$, and analyzing the eigenvalues of the Jacobian matrix. As it is shown in \cite{abuin2020char}, this matrix has one eigenvalue at zero ($\lambda_1 =0$), one always negative ($\lambda_2<0$) and a third one, $\lambda_3$, that is negative, zero or positive depending on if $U_s$ is smaller, equal of greater than $U^*$, respectively.

Since the maximum eigenvalue $\lambda_3$ is the one determining the stability of the system, it is possible to separate set $\setX_s$ into two subsets, according to its behaviour. Then, a first intuition is that the equilibrium subset 
\begin{eqnarray}\label{ec:setXs1}
\setX_s^{st} &:=& \{(U,I,V)\in \setX : U \in [0,U^*],~I=0,~V=0\}
\end{eqnarray}
is stable, and that the equilibrium subset 
\begin{eqnarray}\label{ec:setXs2}
\setX_s^{un} &:=& \{(U,I,V)\in \setX: U \in (U^*,+\infty),~I=0,~V=0\},
\end{eqnarray}
is unstable. However, this is not a conclusive analysis, given that one of the eigenvalues of the linearized system is null and so the linear approximation cannot be used to fully determine the stability of a nonlinear system (Theorem of Hartman-Grobman \cite{perko2013differential}). 
Formal asymptotic stability of set $\setX_s^{st}$, together with its corresponding domain of attraction, is analyzed in the next subsection. 

%%%%%%%%%%%%%%%%%%%%%%%%%%%%%%%%%%%%%%%%%%%%%%%%%%%%%
\subsection{Asymptotic stability of the equilibrium sets}\label{sec:assest}
%%%%%%%%%%%%%%%%%%%%%%%%%%%%%%%%%%%%%%%%%%%%%%%%%%%%%

A key point to properly analyze the asymptotic stability (AS) of system \eqref{eq:SysOrig} is to consider the stability of the equilibrium sets $\setX_s^{st}$ and $\setX_s^{un}$, instead of the single points inside them (as defined in Definitions~\ref{def:attrac_set},~\ref{def:eps_del_stab}
and~\ref{def:AS}, in Appendix~\hyperlink{sec:app1}{1}). Indeed, even when every equilibrium point in $\setX_s^{st}$ is $\epsilon-\delta$ stable, there is no single equilibrium point in such set that is locally attractive. 

As stated in Definition~\ref{def:AS}, in Appendix~\hyperlink{sec:app1}{1}, the AS of $\setX_s^{st}$ requires both, attractivity and $\epsilon-\delta$ stability, which are stated in the next two subsections, respectively. Finally, in Subsection~\ref{sec:ASproof} the AS theorem is formally stated.

%%%%%%%%%%%%%%%%%%%%%%%%%%%%%%%%%%%%%%%%%%%%%%%%%%%%%
\subsection{Attractivity of set $\setX_s^{st}$}\label{sec:attrac}
%%%%%%%%%%%%%%%%%%%%%%%%%%%%%%%%%%%%%%%%%%%%%%%%%%%%%

According to Definition~\ref{def:attrac_set}, any set containing an attractive set is also attractive. So, we are in fact interested in finding the smallest closed attractive set in $\setX\backslash\setX_s^{un}$. 
\begin{theorem}[Attractivity of $\setX_s^{st}$]\label{theo:attract}
	Consider system \eqref{eq:SysOrig} constrained by $\setX$.
	Then, the set $\setX_s^{st}$ defined in \eqref{ec:setXs1} is the smallest attractive set in $\setX\backslash\setX_s^{un}$. Furthermore, $\setX_s^{un}$, defined in \eqref{ec:setXs2}, is not attractive.
\end{theorem}

\begin{proof}
	The proof is divided into two parts. First it is proved that $\setX_s^{st}$ is an attractive set, and then, that it is the smallest one.
	
	\textbf{\textit{Attractivity of $\setX_s^{st}$:}}
	To prove the attractivity of $\setX_s^{st}$ in $\setX$ (and to show that $\setX_s^{un}$ is not attractive) we needs to prove that $U_{\infty} \in [0,U^*]$ for any initial conditions and values of $\Rn$. $U_\infty$ can be expressed as a function of $\Rn$ and initial conditions, as follows
    \begin{equation}\label{eq:Uinfty}
	    U_\infty (\Rn,U(t_0), I(t_0), V(t_0))=-\frac{W(-\Rn U(t_0)e^{-\Rn(U(t_0) + I(t_0) +\frac{\delta}{p}V(t_0))})}{\Rn}
    \end{equation}
    where $W(\cdot)$ is (the principal branch of) the Lambert function and $(U(t_0), I(t_0), V(t_0))$ are arbitrary initial conditions at a given time $t_0\geq0$. The minimum of $U_\infty$ is given by $U_\infty=0$, and it is reached when $U(t_0)=0$ (for any value of $\Rn$, $I(t_0)$ and $V(t_0)$). The maximum of $U_\infty$, on the other hand, is given by $U_\infty=U^*$, and it is reached only when $U(t_0)=U^*$ and $I(t_0)=V(t_0)=0$ (for any value of $\Rn$), as it is shown in Lemma~\ref{lem:Uinf_opt}, in Appendix~\hyperlink{sec:app2}{2}. Then, for any $(U(t_0),I(t_0),V(t_0)) \in \setX$ and $\Rn>0$, $U_\infty (\Rn,U(t_0), I(t_0), V(t_0)) \in [0,U^*]$, which means that $\setX_s^{st}$ is attractive, and the proof of attractivity is complete.
	
	Figure~\ref{fig:Uinfty_U0V0} shows how $U_\infty$ behaves as function of $U(t_0)$ and $V(t_0)$, when $I(t_0)=0$ and $I(t_0)=5e^5$. The first one is the scenario corresponding to $t=0$, when a certain amount of virus enters the healthy host.
    \begin{figure}
	\centering
	\begin{subfigure}{.5\textwidth}
		\centering
		\includegraphics[width=0.95\linewidth]{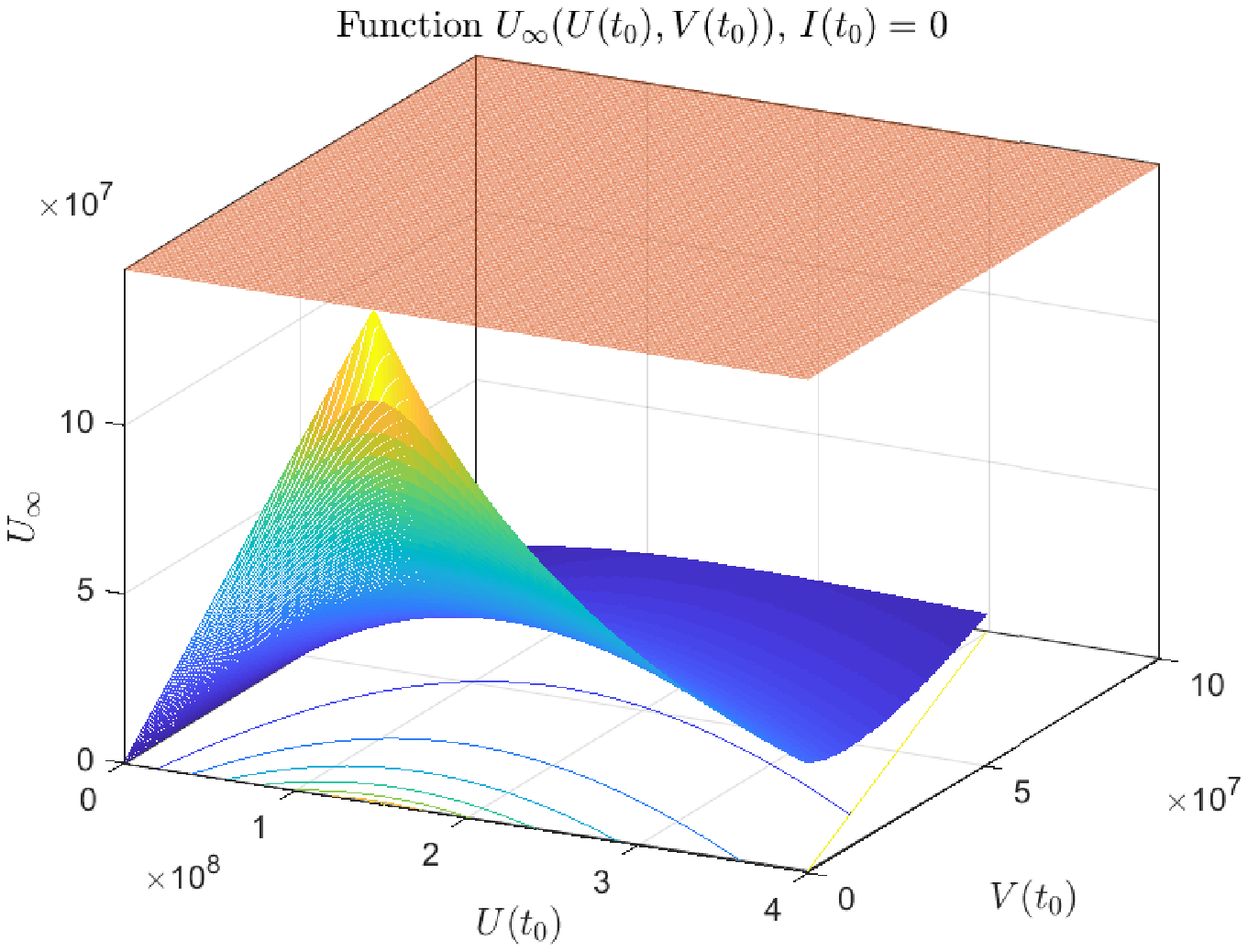}
		\caption{$U_\infty$ as function of $U(t_0)$ and $V(t_0)$, when $I(t_0)=0$. The orange plane represents $U^*=1/\Rn=1.5e^8$. The maximum of $U_\infty$ is reached when $U(t_0) = U^*$ and $V(t_0) =0$, and is given by $U^*$. Patient 'A'.}
		\label{fig:Uinfty_U0V00}
	\end{subfigure}%
	\hspace*{0.2truecm}
	\begin{subfigure}{.5\textwidth}
		\centering
		\includegraphics[width=0.95\linewidth]{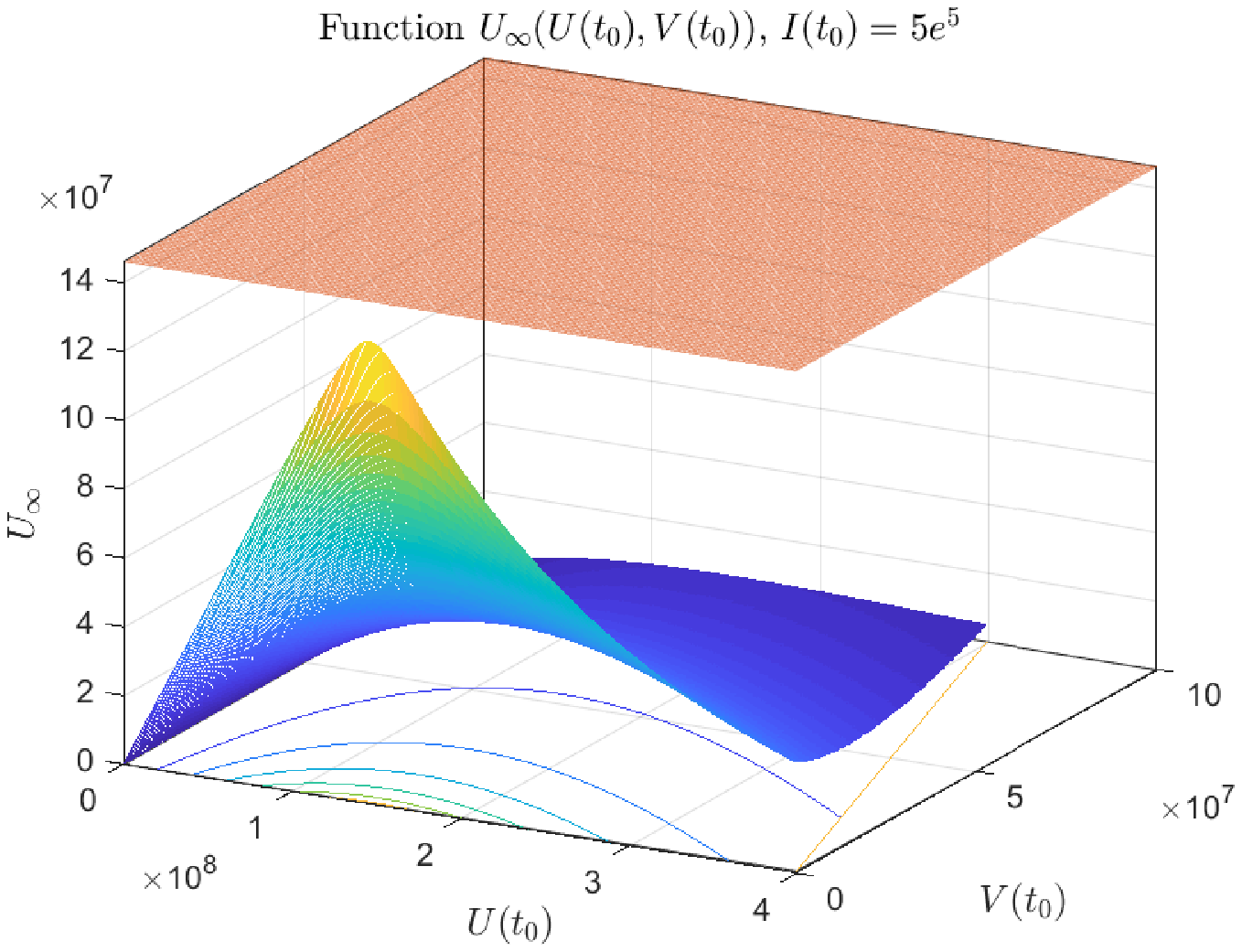}
		\caption{$U_\infty$ as function of $U(t_0)$ and $V(t_0)$, when $I(t_0)=5e^5$. The orange plane represents $U^*=1/\Rn=1.5e^8$. The maximum of $U_\infty$ is reached when $U(t_0) = U^*$ and $V(t_0) = 0$, and is smaller than $U^*$. Patient 'A'.}
		\label{fig:Uinfty_U0V05}
	\end{subfigure}
	\caption{Function $U_\infty(U(t_0),V(t_0))$, for different values of $I(t_0)$.}
	\label{fig:Uinfty_U0V0}
    \end{figure}
    %%%%%%%%%%%
	
	\textbf{\textit{$\setX_s^{st}$ is the smallest attractive set:}}
	It is clear from the previous analysis, that any initial state $x(t_0) =(U(t_0),I(t_0),V(t_0)) \in \setX$ converges to a state $x_{\infty}=(U_{\infty},0,0)$ with $U_{\infty} \in [0,U^*]$. This means that $\setX_s^{un}$ is not attractive for any point in $\setX\backslash\setX_s$.
	However, to show that $\setX_s^{st}$ is the smallest attractive set, we need to prove that every point $x_s\in\setX_s^{st}$ is necessary for the attractiveness.
	
	Let us consider a initial state of the form $(U^*,I(t_0),0)$ with $I(t_0)\ge0$. Since $W$ is a bijective function from $(-1/e,0)$ to $(-1,0)$, then 	$U_\infty (\Rn,U^*, \cdot , 0)$ is bijective from $(0,+\infty)$ to $(0, U^*)$. Hence for every point $x_s\in\inti(\setX_s^{st})$ there exists $I(t_0)\ge0$ such that the initial state $(U^*,I(t_0),0)$ converges to $x_s$.
	Since every interior point of  $\setX_s^{st}$ is necessary for the attractiveness, then the smallest closed attractive set is $\setX_s^{st}$, and the proof is concluded.
\end{proof}

%%%%%%%%%%%%%%%%%%%%%%%%%%%%%%%%%%%%%%%%%%%%%%%%%%%%%%%%%%%%%%%%%%
\subsection{Local $\epsilon-\delta$ stability of $\setX_s^{st}$} \label{sec:epdelstabil}
%%%%%%%%%%%%%%%%%%%%%%%%%%%%%%%%%%%%%%%%%%%%%%%%%%%%%%%%%%%%%%%%%%

The next theorem shows the formal Lyapunov (or $\epsilon-\delta$) stability of the equilibrium set $\setX_s^{st}$.

\begin{theorem}[Local $\epsilon-\delta$ stability of $\setX_s^{st}$]\label{theo:stab}
	Consider system \eqref{eq:SysOrig} constrained by $\setX$.
	Then, the equilibrium set $\setX_s^{st}$ defined in \eqref{ec:setXs1} is the largest locally $\epsilon-\delta$ stable.
\end{theorem}
\begin{proof}
	We proceed by analysing the stability of single equilibrium points $\bar x := (\bar U,0,0)$, with $\bar U \in (0,U_0]$ (i.e., $\bar x \in \setX_s\backslash\{(0,0,0)\}$).
	For each $\bar x$ let us consider the following Lyapunov function candidate
	\begin{eqnarray} \label{ec:lya1}
	J(x) := U-\bar U - \bar U \ln\left(\frac{U}{\bar U}\right) + I + \frac{\delta}{p} V.
	\end{eqnarray}
	This function is continuous in $\setX$, is positive for all nonegative $x \neq \bar x$ and $J(\bar x)=0$.
	Furthermore, for $x(t) \in \setX$ and $t\geq 0$ we have
	\begin{align*} \label{ec:lya2}
	\dot J(x(t))  &= \frac{\partial J}{\partial x} \dot{x}(t) =  \left[\frac{d J}{d U}~~\frac{d J}{d I}~~ \frac{d J}{d V} \right] \left[
	\begin{array}{c}
	-\beta U(t) V(t)  \\
	\beta U(t)V(t)-\delta I(t)  \\
	pI(t) - cV(t)   
	\end{array}\right]\nonumber\\
	&= \left[(1-\frac{\bar U}{U(t)})~~1~~ \frac{\delta}{p} \right] \left[
	\begin{array}{c}
	-\beta U(t) V(t)  \\
	\beta U(t)V(t)-\delta I(t) \\
	pI(t) - cV(t)   
	\end{array}\right]\nonumber\\
	&= (-\beta U(t) V(t) + \bar U \beta V(t)) + (\beta U(t)V(t)-\delta I(t)) + \left(\delta I(t) - \frac{\delta c}{p} V(t)\right)\nonumber\\
	&= \bar U \beta V(t) - \frac{\delta c}{p} V(t) = V(t) \left(\bar U \beta  - \frac{\delta c}{p}\right),
	\end{align*}
where $\dot{x}(t)$ represents system \eqref{eq:SysOrig}. Function $\dot{J}(x(t))$ depends on $x(t)$ only through $V(t)$. So, independently of the value of the parameter $\bar U$, $\dot{J}(x(t))=0$ for $V(t)\equiv 0$. This means that for any single $x(0) \in \setX_s$, $V(0)=I(0)=0$ and so, $V(t)= 0$, for all $t\geq 0$. So $\dot{J}(x(t))$ is null for any $x(0) \in \setX_s$ (i.e, it is not only null for $x(0)=\bar x$ but for any $x(0) \in \setX_s$).
	
On the other hand, for $x(0) \notin \setX_s$, function $\dot{J}(x(t))$ is negative, zero or positive, depending on if the parameter $\bar U$ is smaller, equal or greater than  $U^*= \frac{\delta c}{\beta p}$, respectively, and this holds for all $x(0)\in \setX$ and $t\geq 0$. So, for any $\bar x \in \setX_s^{st}$, $\dot J(x(t))\leq 0$ (particularly, for $\bar x=(\bar U,0,0)=(U^*,0,0)$, $\dot J(x(t)) = 0$, for all $x(0) \in \setX$ and $t\geq 0$) which means that each $\bar x \in \setX_s^{st}$ is locally $\epsilon-\delta$ stable (see Theorem~\ref{theo:lyap} in Appendix~\hyperlink{sec:app1}{1}).

Finally, when $\bar U = 0$, i.e. $\bar x=(0,0,0)$, we define the Lyapunov functional as $J(x) = U-I+\delta/pV$ and we proceed analogously as before to prove the local $\epsilon-\delta$ stability of the origin.

Therefore, since every state in $\setX_s^{st}$ is locally $\epsilon-\delta$ stable and $\setX_s^{st}$ is compact, by Lemma~\ref{lem:stab}, the whole set $\setX_s^{st}$ is locally $\epsilon-\delta$ stable.
	
Finally, since $\setX_s^{st}$ is attractive in $\setX\backslash\setX_s^{un}$ then it is impossible for any $x\in\setX_s^{un}$ to be $\epsilon-\delta$ stable, which implies that $\setX_s^{st}$ is also the largest locally $\epsilon-\delta$ stable set in $\setX_s$, which completes the proof.
\end{proof}

\begin{remark} \label{rem:nonas}
In the latter proof, if we pick a particular $\bar x \in \setX_s^{st}$, then $\dot{J}(x(t))$ is not only null for $x(0)=\bar x$ but for all $x(0) \in \setX_s^{st}$, since in this case, $V(t)=0$, for $t\geq0$. This means that it is not true that $\dot{J}(x(t)) < 0$ for every $x \neq \bar x$, and this is the reason why we cannot use the last part of Theorem~\ref{theo:lyap} to ensure the asymptotic stability of particular equilibrium points (or subsets of $\setX_s^{st}$). In fact, they are $\epsilon-\delta$ stable, but not attractive.
\end{remark}

%%%%%%%%%%%%%%%%%%%%%%%%%%%%%%%%%%%%%%%%%%%%%%%%%%%%%
\subsection{Asymptotic stability of $\setX_s^{st}$} \label{sec:ASproof}
%%%%%%%%%%%%%%%%%%%%%%%%%%%%%%%%%%%%%%%%%%%%%%%%%%%%%

In the next Theorem, based on the previous results concerning the attractivity and $\epsilon-\delta$ stability of $\setX_s^{st}$, the asymptotic
stability is formally stated.

\begin{theorem}\label{theo:AS}
	Consider system \eqref{eq:SysOrig} constrained by the positive set $\setX$.
	Then, the set $\setX_s^{st}$ defined in \eqref{ec:setXs1} is the unique asymptotically stable (AS) equilibrium set, with a domain of attraction (DOA) given by $\setX\backslash\setX_s^{un}$. Furthermore, $\setX_s^{un}$ is unstable.
\end{theorem}

\begin{proof}
	The proof follows from Theorems~\ref{theo:attract}, which states that $\setX_s^{st}$ is the smallest attractive in $\setX$, and~\ref{theo:stab},
	which states that $\setX_s^{st}$ is the largest locally $\epsilon-\delta$ stable set in~$\setX$. 
\end{proof}

Figures~\ref{fig:FhasePortNoSpread} shows phase portrait plots of system \eqref{eq:SysOrig}, corresponding to different initial conditions.
\begin{figure}
	\centering
	\begin{subfigure}{.5\textwidth}
		\centering
		\includegraphics[width=1\linewidth]{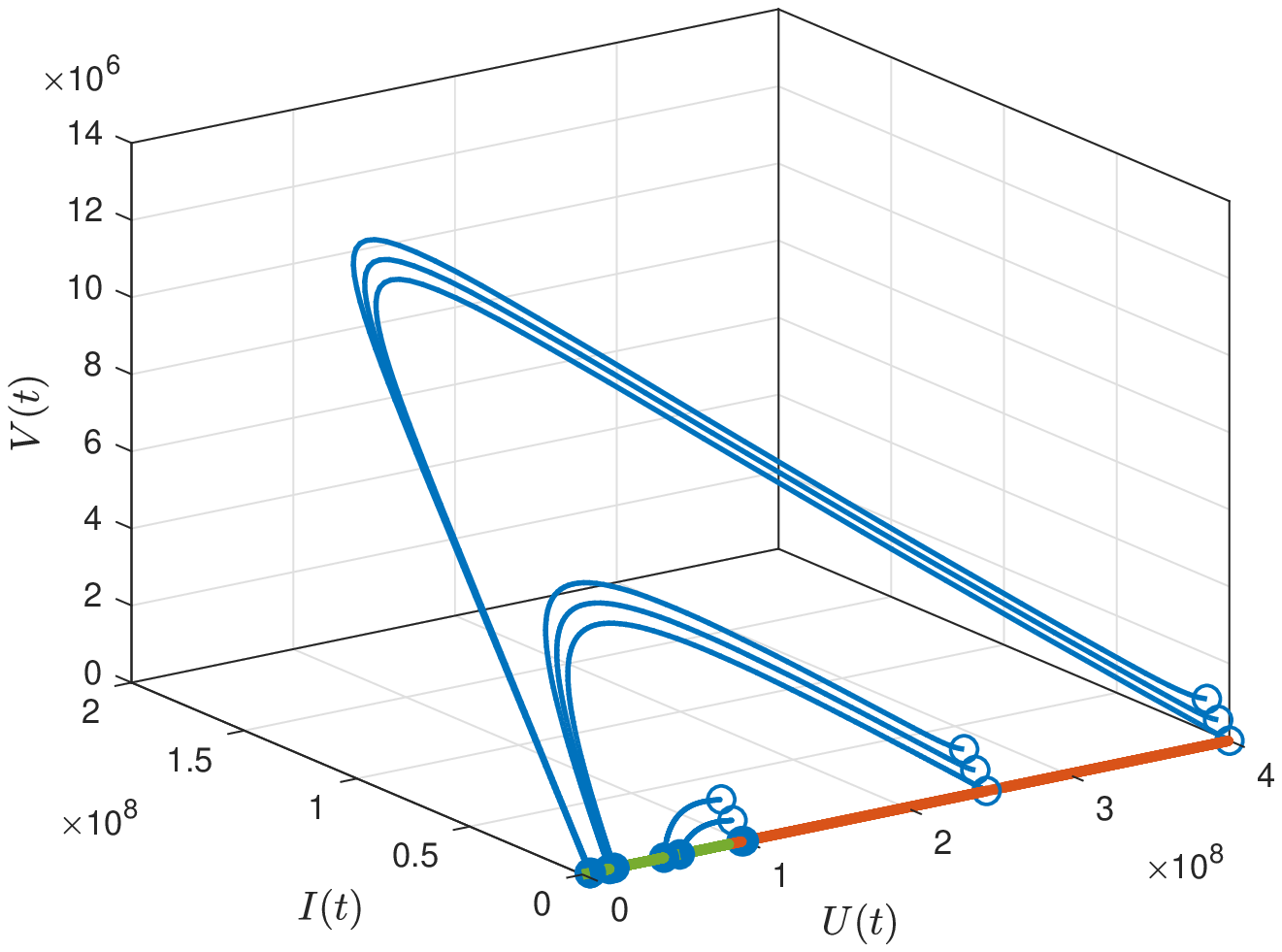}
		\caption{}
		\label{fig:FhasePortSpread}
	\end{subfigure}%
	\hspace*{0.2truecm}
	\begin{subfigure}{.5\textwidth}
		\centering
		\includegraphics[width=1\linewidth]{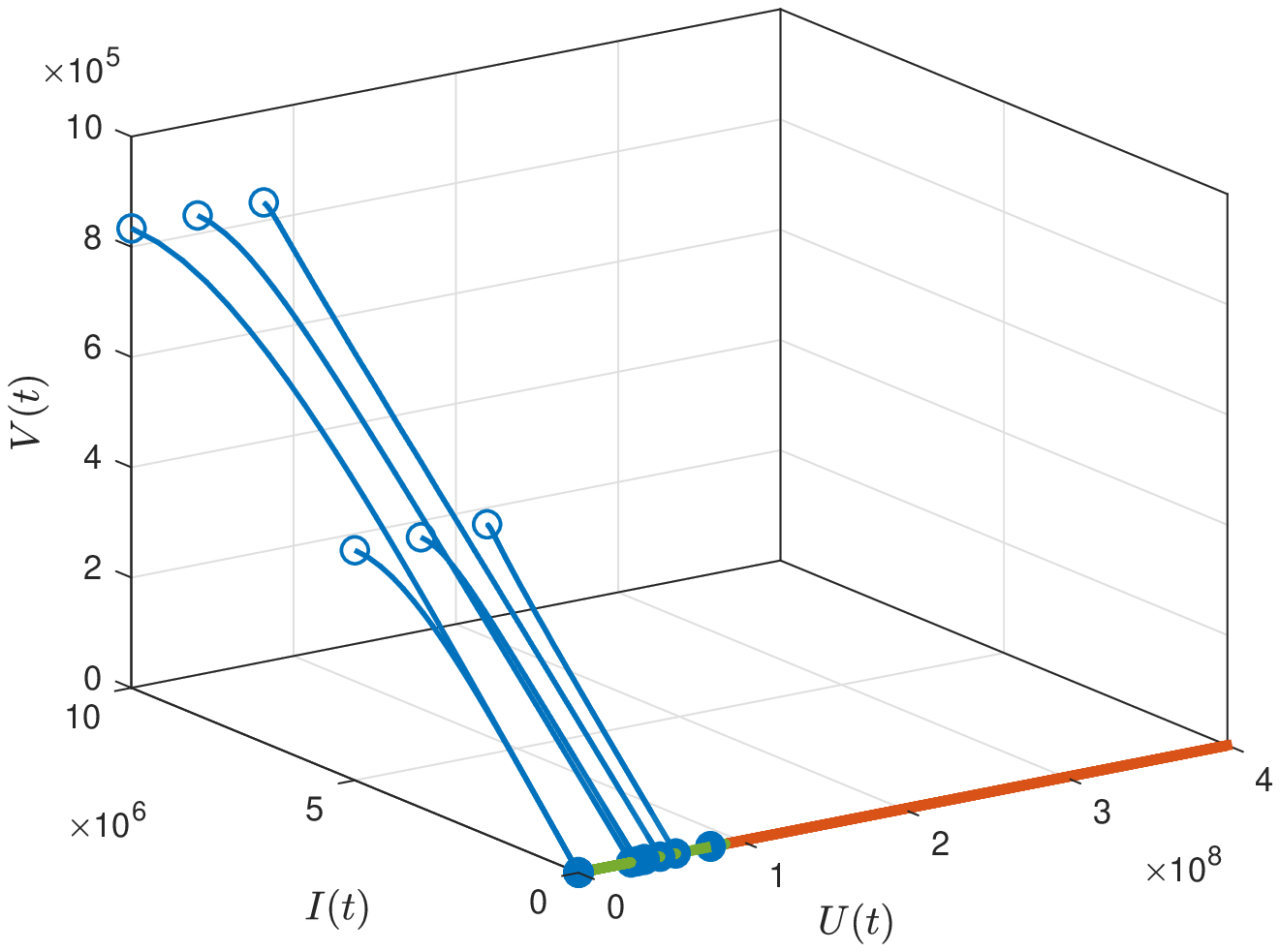}
		\caption{}
		\label{fig:sub2}
	\end{subfigure}
	\caption{Phase portrait for virtual patient 'A', described in Section~\ref{sec:IntPatA}. (a) Case $U(t_0)>U^*$. States arbitrarily close to $\setX_s^{un}$ (in red), converges to $\setX_s^{st}$ (in green), so the virus spreads in the host. (b) Case $U(t_0)<U^*$. States arbitrarily close to $\setX_s^{st}$, converges to $\setX_s^{st}$, so the virus does not spread in the host. Empty circles represent the initial state, while solid circles represent final states.}
	\label{fig:FhasePortNoSpread}
\end{figure}
%%
%%%%%%%%%%%%%%%%%%%%%%%%%%%%%%%%%%%%%%
\subsection{$U_\infty$ as function of initial conditions}\label{sec:InitCond}
%%%%%%%%%%%%%%%%%%%%%%%%%%%%%%%%%%%%%%

In this section some characteristics of system \eqref{eq:SysOrig} concerning the value of $U_\infty$ as a function of the reproduction number $\Rn$ and the initial conditions are analyzed. Consider the next Property.
\begin{property}\label{propt:sinfty}
    Consider system \eqref{eq:SysOrig} with arbitrary initial conditions $(U(t_0),I(t_0),V(t_0)) \in \setX$, for some $t_0 \geq 0$. Then:
\begin{enumerate}
%i
\item For any value of $U(t_0)>0$, $I(t_0)>0$, $V(t_0)>0$, $U_\infty(\Rn,U(t_0),I(t_0),V(t_0)) \rightarrow 0$, when $\Rn \rightarrow \infty$; while $U_\infty(\Rn,U(t_0),I(t_0),V(t_0))$ remains close to $U(t_0)$ when $\Rn \rightarrow 0$.
%ii
\item For $U(t_0) > U^*$ and fixed $I(t_0)>0$, $V(t_0)>0$ and $\Rn>0$, $U_\infty(\Rn,U(t_0),I(t_0),V(t_0))$ decreases when $U(t_0)$ increase, and $U_\infty(\Rn,U(t_0),I(t_0),V(t_0)) < U^*$. This means that the closer $U(t_0)$ is to $U^*$ from above, the closer will be $U_\infty$ to $U^*$ from below.
%iii
\item For $U(t_0) < U^*$ and fixed $I(t_0)>0$, $V(t_0)>0$ and $\Rn>0$, $U_\infty(\Rn,U(t_0),I(t_0),V(t_0))$ increases with $U(t_0)$, and $U_\infty(\Rn,U(t_0),I(t_0),V(t_0)) < U^*$. This means that smaller values of $U(t_0)$ produce smaller values of $U_\infty$, both below $U^*$.
%iv
\item For any fixed $U(t_0)$ and $\Rn>0$, $U_\infty(\Rn,U(t_0),I(t_0),V(t_0))$ decrease with $I(t_0)$ and $V(t_0)$, and $U_\infty(\Rn,U(t_0),I(t_0),V(t_0)) \leq U^*$.
%v
\item For fixed $\Rn>0$, $U(t_0) = U^*$ and $I(t_0)=V(t_0)=0$, $U_\infty(\Rn,U(t_0),I(t_0),V(t_0))$ reaches its maximum over $\setX$, and the maximum value is given by $U^*$ (see Lemma~\ref{lem:Uinf_opt} in Appendix~\hyperlink{sec:app2}{2}). 
\end{enumerate}
\end{property} 

The proof of the properties are omitted for brevity. However, Figures~\ref{fig:Uinfty_U0V0} and~\ref{fig:Uinf_U} show how $U_\infty$ behaves for different values of initial conditions.

%%%%%%%%%%%%%%%%%%%%%%%%%%%%%%%%%%%%%%
\subsection{Simulation example}\label{sec:IntPatA}
%%%%%%%%%%%%%%%%%%%%%%%%%%%%%%%%%%%%%%

All along this work we use a virtual patient, denoted as patient 'A', to demonstrate the results of each section. The parameters of patient 'A' were estimated by
using viral load data of a RT-PCR COVID-19 positive patient ---reported in \cite{wolfel2020virological} and used in \cite{abuin2020char,vargas2020host}--- and are given by 
\begin{table}[H]
	\begin{center}		
		\caption{Target cell-limited model parameters for COVID-19, patient A \cite{vargas2020host}}
		\begin{tabular}{| c | c | c |c |}
			\hline
			  $\beta$ &  $\delta$ &  $p$  & $c$ \\\hline
			  $1.35\times10^{-7}$  &  0.61 &  0.2  &  2.4 \\ \hline
		\end{tabular}
		\label{tab:param}
	\end{center}
\end{table}
The initial conditions are given by: $U_0=4 \times 10^8$, $I_0=0$ and $V_0=0.31$. Furthermore, the reproduction number is $\Rn = 1.84 \times 10^{-8}$, while the critical value for the susceptible cells is $U^* = 5.44 \times 10^7$. The final value of $U$ (if no antiviral treatment is applied) is given by $U_{\infty} =2.57 \times 10^5 $, which means that the area under the curve (AUC) of $V$ is given by $AUC_V = 5.45 \times 10^7$. The peak of $V$ is given by $\hat V=1.98 \times 10^7$.
Figure~\ref{fig:OL} shows the time response corresponding to patient 'A'. As predicted, $U_{\infty}$ is (significantly) smaller than $U^*$, which means that antivirals reducing (even for a finite period of time) either $p$ or $\beta$ will increase $U_{\infty}$ and, so, will reduce the AUC and, probably, the peak of $V$.

\begin{remark}
    Note that the area under the curve of $V$, between times $t_{d1}$ and $t_{d2}$ is given by $AUC_V:=\int_{t_{d1}}^{t_{d2}} V(t) d t=\frac{1}{c}[\frac{p}{\delta}(U(t_{d1})-U(t_{d2}) + I(t_{d1})-I(t_{d2})) + V(t_{d1}) - V(t_{d2})]$. Therefore, assuming  $U(t_{d1})=U(t_0)$, $I(t_{d1})=I(t_0)$, $V(t_{d1})=V(t_0)$, with $U(t_0)\gg I(t_0)$, $U(t_0)\gg V(t_0)$, and $U(t_{d2})=U_\infty$, $I(t_{d2})=0$ and $V(t_{d2})=0$, which gives: $AUC_V \approx \frac{1}{c}[\frac{p}{\delta}(U(t_0)-U_\infty)]$. This way, if $U_\infty$ is increased with respect to the value corresponding to the untreated case, the AUC of viral load decreases. Moreover, as it was shown in \cite{abuin2020dynamical}, the viral load at time to peak is monotonically decreasing with antiviral therapy reducing $\beta$ or $p$.
\end{remark}

\begin{figure}
	\centering
	\includegraphics[width=0.75\textwidth]{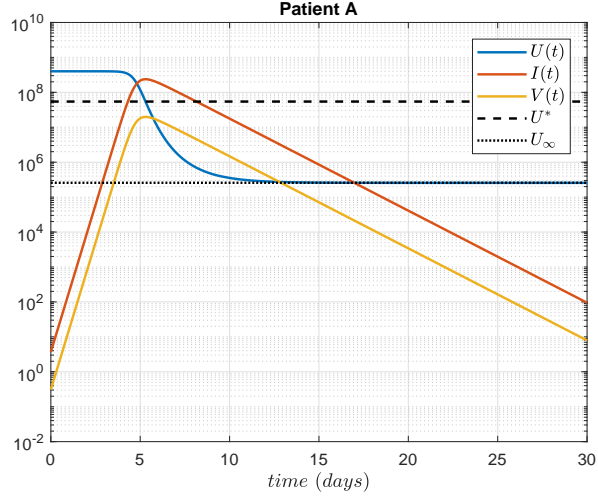}
	\caption{\small{Time evolution of virtual patient 'A'. As $c>>\delta$ (as it is always the case for real patient data), $I(t) \approx \frac{c}{p}V(t)$ for  all $t\geq 0$.}}
	\label{fig:OL}
\end{figure}
%

%%%%%%%%%%%%%%%%%%%%%%%%%%%%%%%%%%%%%%%%%%%%%%%%%%%%%
\section{Inclusion of PK and PD of antiviral treatment} \label{sec:pkpd}
%%%%%%%%%%%%%%%%%%%%%%%%%%%%%%%%%%%%%%%%%%%%%%%%%%%%%

The idea now is to formally incorporate the pharmacodynamic (PD) and pharmacokinetics (PK) of antivirals into system~\ref{eq:SysOrig}, to obtain a controlled system, i.e. a system with certain control actions - given by the antivirals - that allows us to (even partially) modify the whole system dynamic according to some control objectives. In contrast to vaccines that kill the virus, antiviral just inhibits the virus infection and replication rates, so reducing the advance of the infections in the respiratory tract. The PD is introduced in system \eqref{eq:SysOrig} as follows:
%\todo{Ale. Por ahora dej\'e solo el efecto en $\beta$. Si metemos el $p$, vamos a tener dos controles, que afectan igual al $\Rn$, y vamos a tener una comibnaci\'on de controles \'optmios. Hay que ver si vale la pena complicarla o no.}
%
\begin{subequations}\label{eq:SysOrigPD}
	\begin{align}
		&\dot{U}(t)   =  - \beta(1 - \eta(t)) U(t) V(t), \\
		&\dot{I}(t) =  \beta(1 - \eta(t)) U(t) V(t) - \delta I(t) \\
		&\dot{V}(t) =  p I(t) - c V(t),
	\end{align}
\end{subequations}
where $\eta(t) \in [0,1)$ represents the inhibition antiviral effects affecting the infection rate $\beta$ (note that, according to \cite{abuin2020dynamical}, the effect of antivirals on the replication rate $p$, is analogous to the one on $\beta$, since both parameters affect in the same way the reproduction number $\Rn$). 

On the other hand, the PK is modeled as a one compartment with an impulsive input action (to properly account for pills intakes or injections):
%\todo{Se podr\'ia usar una din\'amica m\'as compleja para las drogas}
%
\begin{subequations}\label{eq:PK}
\begin{align}
&\dot{D}(t) = - \delta_D D(t),~~ t \neq t_k,  \\
&D(t_k) = D(t_k^-) + u_{k-1},~~ k \in \I, 
\end{align}
\end{subequations}
where $D$ is the amount of drug available (with $D(0)=D_0=0$), $\delta_D$ is the drug elimination rate and the antiviral dose $u_k$ enters the system impulsively at times $t_k:=kT$, with $T>0$ being a fix time interval and $k \in \I$. Time $t_k^-$ denotes the time just before $t_k$, i.e., $D(t_k^-) = \lim_{\delta \rightarrow 0^+} D(t_k - \delta)$. Note that \eqref{eq:PK} is a continuous-time system impulsively controlled, which shows discontinuities of the first kind (jumps) at times $t_k$ and free responses in $t \in [t_k,t_{k+1})$ (see \cite{rivadeneira2017control} for details). 
 
Finally, the way the drug $D$ enters system \eqref{eq:SysOrigPD} is by means of $\eta$ as follows:
\begin{eqnarray} \label{eq:PD}
\eta(t) &=& \frac{D(t)}{D(t)+EC_{50}}
\end{eqnarray}
where $EC_{50}$ represents the drug concentration in the blood where the drug is half-maximal. $\eta(t)$ is assumed to be in $[0,\eta_{\max})$, with $\eta_{\max}<1$ (not full antiviral effect is considered, since this is an unrealistic scenario). 

%%%%%%%%%%%%%%%%%%%%%%%%%%%%%%%%%%
\subsection{Impulsive scheme}\label{sec:impsys}
%%%%%%%%%%%%%%%%%%%%%%%%%%%%%%%%%

Based on the PK and PD previous analysis, the complete Covid-19 infection model, taking into account an antiviral treatment (the controlled system) reads as follows:
\begin{subequations}\label{eq:SysOrigPKPD}
	\begin{align}
	&\dot{U}(t)   =  - \beta(1 - \eta(t)) U(t) V(t),~~ t \neq t_k, \\
	&\dot{I}(t) =  \beta(1 - \eta(t)) U(t) V(t) - \delta I(t),~~ t \neq t_k, \\
	&\dot{V}(t)   = p I(t) - c V(t),~~ t \neq t_k, \\
	&\dot{D}(t) = - \delta_D D(t),~~ t \neq t_k,\\
	&D(t_k) = D(t_k^-) + u_{k-1},~~ k \in \I
	\end{align}
\end{subequations}
with initial conditions given by $x_0=(U_0,I_0,V_0,D_0)$. Given that $D(t)\geq 0$ for all $t\geq 0$, the constraint set $\setX$ is enlarged to be $\tilde \setX := \R^4_{\geq 0}$. Also, a constraint for the input, $u$, is defined as $\tilde \U:=\{u\in\R: 0 \leq u \leq u_{\max} \}$, where $u_{\max}$ represent the maximal antiviral dosage ($u_{\max}$ is usually determined by the drug side effects and maximal effectivity, $0\leq \eta_{\max}<1$), while sets $\setX$ is enlarged by considering $\tilde \setX:=\setX \times \R_{\geq0}$. A detailed study of the stability of impulsive systems can be seen in \cite{djorge2020stability}.

%%%%%%%%%%%%%%%%%%%%%%%%%%%%%%%%%%%%%%%%%%%%%%%%%%%
\subsection{Simulation example}\label{sec:simex}
%%%%%%%%%%%%%%%%%%%%%%%%%%%%%%%%%%%%%%%%%%%%%%%%%%%

We resume here the simulation of the virtual patient 'A', to demonstrate the impulsive control actions describing the effects of antiviral administration. It is assumed that antivirals affect the infection rate $\beta$, while the initial condition for $D$ is $D_0=0$, $\delta_D = 2$ (days$^{-1}$) and $EC_{50}^p = 75$ (mg). A scenario of $30$ days was simulated, and a permanent dose of $u_k=20$ (mg) of antivirals is administered each $T$ days, starting at $t_i=4$ days, with $T=1$, $T = 2$ and $T=0.5$. As shown in Figures~\ref{fig:OLImp1} -~\ref{fig:OLImp2}, the system response is quite different for different sampling times. For $T=1$ days, the antiviral treatment is able to decreases $V$ from the beginning. On the other hand, for $T = 2$, the treatment is unable to stop the spread of virus ($V$ continue increasing after the treatment is initiated), as it is shown in Figure~\ref{fig:OLImp2}. Clearly, the effect of larger values of $T$ is equivalent to smaller values of the dose $u_k$. In what follows, for the sake of clarity, $T$ will be fixed in $1$ day and only the (constant) value of the doses $u_k$ - together with the initial and final time of the treatment - will be modified to analyze the different outcomes.
\begin{figure}
	\centering
	\begin{subfigure}{.5\textwidth}
		\includegraphics[width=1\textwidth]{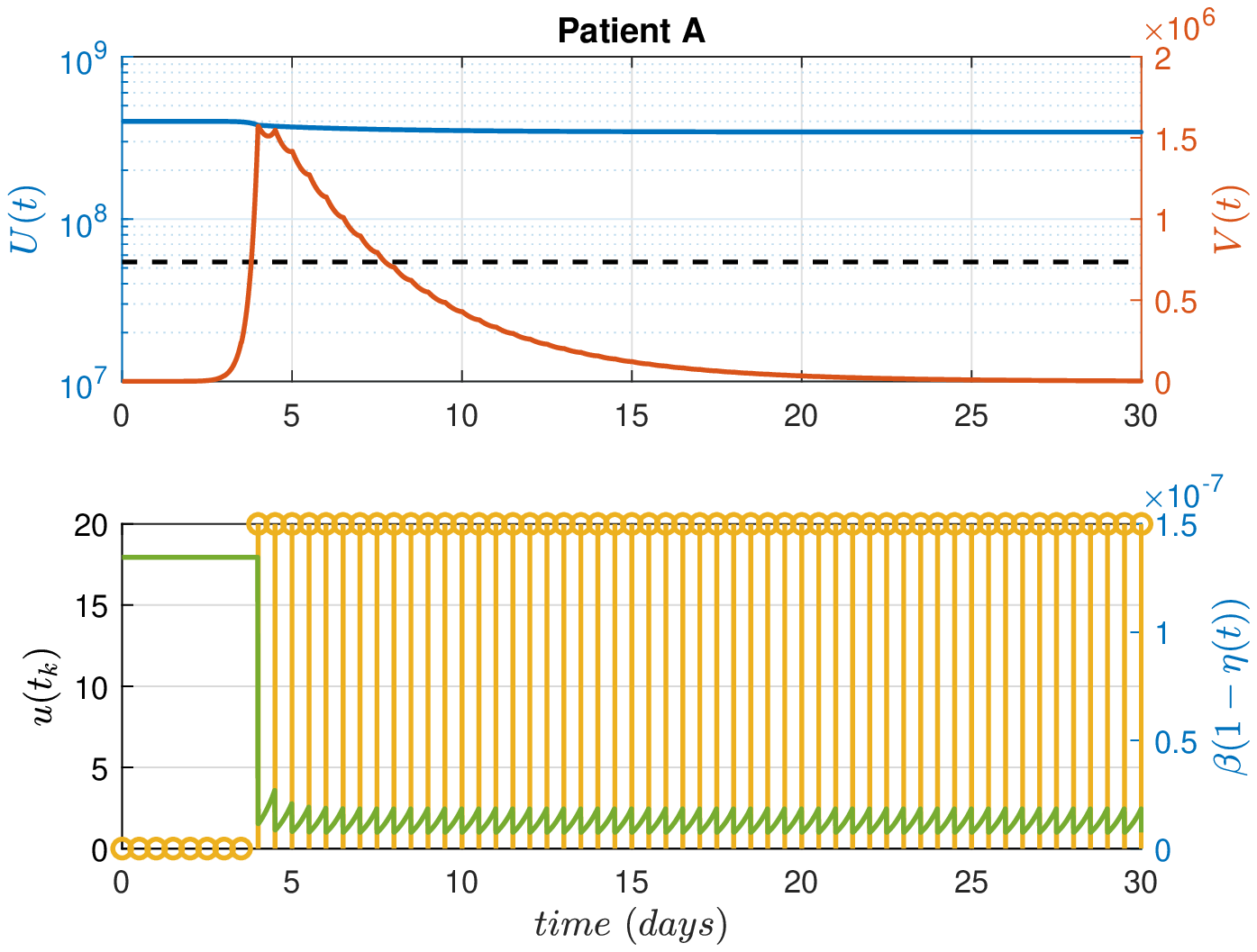}
	\caption{}
	\label{fig:OLImp1}
	\end{subfigure}%
	\hspace*{0.2truecm}
	\begin{subfigure}{.5\textwidth}
			\centering
	\includegraphics[width=1\textwidth]{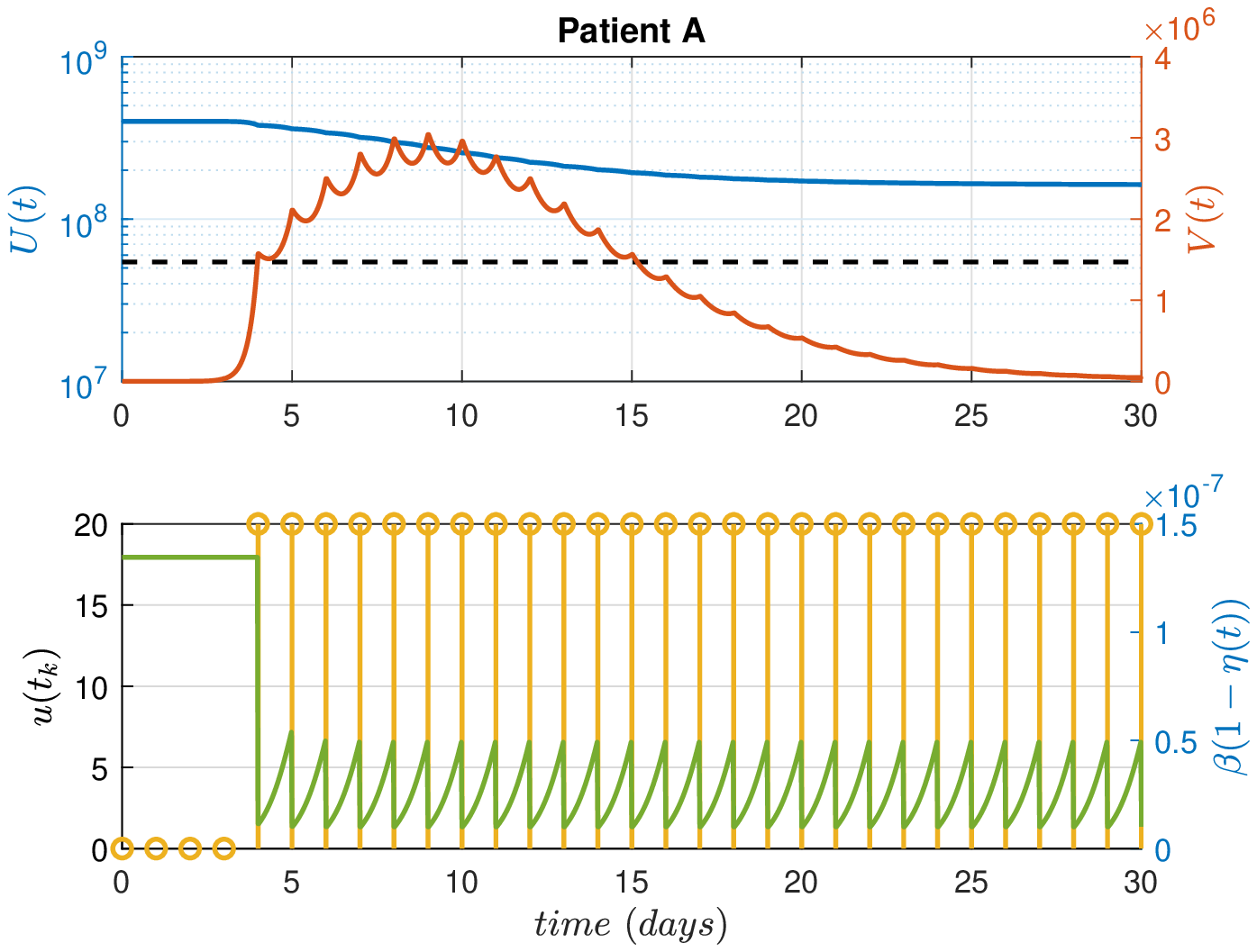}
	\caption{}
	\label{fig:OLImp2}
	\end{subfigure}
	\caption{ (a) Time evolution of virtual patient 'A', with $u=20$ mg of antivirals and $T=1$ days. (b) Time evolution of virtual patient 'A', with $u=20$ mg of antivirals and $T=2$ days.}
\end{figure}
%

%%%%%%%%%%%%%%%%%%%%%%%%%%%%%%%%%%%%%%%%%%%%%%%%%%%
\section{Control}\label{sec:control}
%%%%%%%%%%%%%%%%%%%%%%%%%%%%%%%%%%%%%%%%%%%%%%%%%%%

Control objectives in 'in host' infections can be defined in several ways. The peak of the virus load uses to be a critical index to minimize, since it is directly related to the severity of the infection and the ineffective capacity of the host. However, other indexes - usually put in a second place - are also important. This is the case of the time the infection lasts in the host over significant levels \cite{abuin2020dynamical} - including virus rebounds after reaching a pseudo steady state, and the total viral load or infected cells at the end of the infection (i.e., the AUC of $V$ and $I$). These latter indexes also informs (in a different manner) about the severity of the infection and the time during which the host is able to infect other individuals, and are directly determined by the amount of susceptible cells at the end of the infection. So the twofold control objective is defined as follows:
\begin{definition}[Control objectives]
    The control objective for the closed-loop \eqref{eq:SysOrigPKPD} consists in both, maximize the final value of susceptible/uninfected cells at the end of the infection, $U_\infty$ and minimize the virus peak, $\hat V$. We denote these objectives a Objective 1 and 2, respectively. 
\end{definition}

As it was said in the previous Section, antivirals affect the infection rate $\beta$, by the time-variant factor $(1-\eta(t))$. Accordingly, the reproduction number $\Rn$ will be also time varying, following the formula:
\begin{eqnarray} \label{eq:R_t}
	\Rn(t)   :=  \frac{\beta(1-\eta(t))p}{c \delta},
\end{eqnarray}
and the original reproduction number - i.e., the one corresponding to no treatment - will be denoted as $\Rn(0)$ for clarity ($\Rn(0)$ is the reproduction number at the outbreak of the infection, when $u_0=0$ and $\eta=0$).

We will assume in the following a single interval antiviral treatment, consisting in a single fixed dose of antiviral, applied during a finite period of time. At the outbreak of the infection ($t=0$), it is $(U(0),I(0),V(0)):=(U_0,0,\epsilon)$, with $\epsilon >0$ arbitrary small. Then, the single interval treatment is defined by the following input function:
\begin{eqnarray} \label{eq:contact}
	u_k=u(t_k) = \left\{ 
	\begin{array}{cc}
		0  &  \mbox{for}~ t_k \in [0,t_i),   \\
		u_i   & \mbox{for}~ t_k \in [t_i,t_f],\\
		0  & \mbox{for}~ t_k \in (t_f,\infty).
		\end{array} \right.
\end{eqnarray}
where $t_i < \hat t(\Rn(0))$, being $\hat t(\Rn(0))$ the time of the peak of $V(\tau)$ when no treatment is implemented, $u_i \in [0,u_{\max}]$, and $t_f>t_i$, but finite. Note that after $t_f$, $u(t_k)=0$, which means that $\eta(t) \rightarrow 0$ and $\Rn \rightarrow \Rn(0)$.

%%%%%%%%%%%%%%%%%%%%%%%%%%%%%%%%%%%%%%%%%%
\subsection{First control objective: maximizing the final value of the uninfected cells}
%%%%%%%%%%%%%%%%%%%%%%%%%%%%%%%%%%%%%%%%%%

The control problem we want to solve first reads as follows: for a given initial time, $t_i < \hat t$, find $u_i$ (which has an associated $\Rn_i$, and $\eta_i$) and $t_f$ (finite) to maximize $U_\infty=U_\infty(\Rn(0),U(t_f),I(t_f),V(t_f))$. This control problem accounts for the first control objective; next, somme comments will be made concerning the second one.

A critical point concerning antiviral treatments - that is usually disregarded - is that they are always transitory control actions, not permanent ones. It is not possible to maintain a given treatment for a long time, and its interruption must be explicitly considered in any antiviral schedule.

So, according to the stability results from the previous sections, the following Property holds:

\begin{property}[Upper bound for $U_\infty$] \label{propt:maxss}
	Consider system \eqref{eq:SysOrigPKPD} with $U(0)>U^*$ (or $U(0)\Rn(0)>1$). No matter which kind of antiviral treatment is implemented at time $t_i$, if it is interrupted at some finite time $t_f>t_i$ (as it is always the case), the system converges to an equilibrium state $(U_\infty,0, 0)$ with $U_\infty \leq U^*$, being $U^*$ the critical value for $U$ corresponding to no antiviral treatment, i.e., $U^* = 1/\Rn(0)$.  %\marce{acá cambié el menor por menor igual. Después lo charlamos}
\end{property}
\begin{proof}
We proceed by contradiction. Assume that $U_\infty> U^*$. Consider system \eqref{eq:SysOrigPD} for $t \geq t_f$. Since the antiviral treatment is interrupted at time $t_f$, then  $\eta(t)\searrow 0$, for $t \geq t_f$\footnote{We use the symbols $\searrow$ and $\nearrow $ to indicate that the convergence is monotonic, decreasing and increasing, respectively.}. By eq.~\eqref{eq:R_t} we have that $\Rn(t) \nearrow \Rn(0)$, for $t \geq t_f$. If we denote $U_c(t) = 1/\Rn(t)$ we have $U_c\searrow U^*$ for $t \geq t_f$.  Since $U_\infty> U^*$ then $U_\infty>U_c(t)$ for $t>T$, with $T$ large enough. Hence $(U(t),I(t),V(t))$ is converging to an equilibrium point in the unstable part, which is a contradiction. Therefore $U_\infty\le U^*$, which concludes the proof.  %\marce{le sigue faltando precisión, pero creo que para el chapter está}
\end{proof}

\begin{remark}
	From a clinical perspective, what Property~\ref{propt:maxss} establishes is more than a simple upper bound for $U_\infty$. It says that the best an antiviral treatment can do in terms of the \textbf{total amount of virus (or infected cells) at the end of the infection}, $V_{\tot}:=\int_{t=0}^{\infty} V(t) d t \approx \int_{t=0}^{\infty} \frac{p}{c}I(t) d t$, is to reach a minimal value intrinsically determined by the system parameters ($\Rn(0)$). Furthermore, the instantaneous \textbf{peak of $V(t)$}, for $t>0$, which is the other critical index for the severity of the infection (whose minimization is the second control objective), is independent of the latter lower limit (as shown later on), and can be minimized while maintaining $V_{\tot}$ at its minimal value.
	This represents a new paradigm concerning what (and what not) antiviral treatments can do in acute infections.
\end{remark}

In the search of such a value, the next definition is stated.
\begin{definition}[Goldilocks antiviral dose]
	The goldilocks antiviral dose (GAD), $u^g=u^g(t_i)$, is the one that, if applied at $t_i < \hat t(\Rn(0))$, produces $U_\infty(\Rn^{g},U(t_i),I(t_i),V(t_i)) = U^*$, where $\Rn^g$ is determined by $u^g$, at steady state\footnote{$\Rn^g := \frac{\beta(1-\eta^g(t))p}{c \delta}$, is assumed to be fixed, for simplicity, even when we know that $\eta^g(t)$ is periodic.}.
\end{definition}
%

%%%%%%%%%%%%%%%%%%%%%%%%%%%%%%%%%%%%%%%%%%%%%%%%%%%%%%%%%%%%%%%%%%%%%%%%%%%%%%%%%%%%%%%%%%
\begin{remark}[$u^g$ computation]
	Given $t_i$ and $\Rn(0)$, $u^{g} =u^{g}(U(t_i),I(t_i),V(t_i))=u^{g}(t_i)$ can be obtained, numerically, by means of Algorithm~\ref{alg:uopt}.
\end{remark}
\begin{algorithm}%[H]
	\SetAlgoLined
	%\KwResult{Write here the result }
	$u_k=0$, $U^*=1/\Rn(0)$\;
	Compute $U_i$, $I_i$ and $V_i$ by integrating system \eqref{eq:SysOrig} from $0$ to $t_i$, starting at $(U_0,I_0,V_0)$\;
	Compute $U(t_f)$ by integrating system \eqref{eq:SysOrigPKPD} form $t_i$ to $t_f$, starting at $(U_i,I_i,V_i)$, with $u_k$\; 
	\While{$U(t_f) <= U^*$}{
	    $u_k = u_k + 0.001$\;
		Compute $U(t_f)$ by integrating system \eqref{eq:SysOrigPKPD} form $t_i$ to $t_f$, starting at $(U_i,I_i,V_i)$, with $u_k$\;
	}
	$u^{g}=u_k$\;
	\caption{Computation of $u^{g}(t_i)$}
	\label{alg:uopt}
\end{algorithm}
%%%%%%%%%%%%%%%%%%%%%%%%%%%%%%%%%%%%%%%%%%%%%%%%%%%%%%%%%%%%%%%%%%%%%%%%%%%%%%%%%%%%%%%%%%

Clearly, Goldilocks antiviral treatment cannot be applied indefinitely, since $t_f$ is finite. However, it can be applied up to a time $t_f$ large enough such that $(U(t_f), I(t_f), V(t_f))$ is arbitrarily close to $(U^*, 0, 0)$ from above. This latter scenario is denoted as quasi steady state (QSS), and it allows us to introduce the following definition.
\begin{definition}[Quasi optimal single interval antiviral treatment]\label{teo:cont_act}
	Consider a given starting time, $t_i \in (0,\hat t(\Rn(0)))$. Then, the quasi optimal single interval antiviral treatment consists in applying $u^{g}$, up to a time $t_f$ large enough for the the system to reach a QSS condition (i.e., $U(t_f) \approx U^*$ $I(t_f) \approx 0$, $V(t_f)\approx 0$).
\end{definition}
\begin{remark}
	Clearly, the latter definition refers to a quasi optimal single interval control action, because larger values of $t_f$ will produce values of $U(t_f)$, $I(t_f)$ and $V(t_f)$ closer to $U^*$, $0$ and $0$, respectively, so $U_\infty(\Rn(0),U(t_f),I(t_f),V(t_f))$ will be closer to $U^*$. 
\end{remark}
The next Theorem, which is one of the main contribution of the work, summarizes the latter results by means of a classification that consider every possible single interval treatment case.

\begin{theorem}[Single interval antiviral treatment scenarios]\label{teo:cont_sce}
	Consider system \eqref{eq:SysOrig} with initial conditions $(U(0),I(0),V(0))=(U_0,0,\epsilon)$, with $\epsilon >0$ arbitrary small, and $\Rn(0)$ such that $U(0)>U^*$. Consider also single interval antiviral treatment (as the one defined in \eqref{eq:contact}), with a given starting time $t_i \in (0,\hat t(\Rn(0)))$, and a finite final time $t_f$. Define soft and strong treatments depending on if $u_i<u^{g}$ or $u_i>u^{g}$, respectively. Define also long and short term treatments depending on if the system reaches or does not reach a QSS at $t_f$. Then, the following scenarios can take place:
	\begin{enumerate}
	\item Quasi optimal single interval antiviral treatment: if $u_i=u^{g}$, and $t_f$ is such that $(U(t_f),I(t_f),V(t_f))$ reaches a QSS, then $U_\infty(\Rn(0),U(t_f),I(t_f),V(t_f)) \approx U^*$. Furthermore, the closer is $U(t_f)$ to $U^*$ (or $I(t_f)$ and $V(t_f)$ to zero), the closer will be $U_\infty(\Rn(0),U(t_f),I(t_f),V(t_f))$ to $U^*$.
	\item Soft long-term antiviral treatment: if $(U(t_f),I(t_f),V(t_f))$ reaches a QSS, and $U(t_f) < U^*$, then $U_\infty(\Rn(0),U(t_f),I(t_f),V(t_f)) \approx U(t_f) <U^*$; \textit{i.e.}, $U(t)$ will remain approximately constant for $t \geq t_f$. Furthermore, the softer a soft long term antiviral treatment is, the smaller will be $U_\infty(\Rn(0),U(t_f),I(t_f),V(t_f))$.
	\item Strong long-term antiviral treatment: if $(U(t_f),I(t_f),V(t_f))$ reaches a QSS, and $U(t_f) >U^*$, a \textbf{second outbreak} wave will necessarily take place at some time $\hat{\hat{t}} > t_f$ and, finally, the system will converge to an $U_\infty(\Rn(0),U(t_f),I(t_f),V(t_f))< U^*$. Furthermore, the stronger a strong long term antiviral treatment is, the larger will be the second wave and the smaller will be $U_\infty(\Rn(0),U(t_f),I(t_f),V(t_f))$. 
	\item Short-term antiviral treatment: if $(U(t_f),I(t_f),V(t_f))$ does not reach a QSS (\textit{i.e}, if $V(t_f) \not\approx 0$), then soft, strong and Goldilocks dose will necessarily produce values of $U_\infty(\Rn(0),U(t_f),I(t_f),V(t_f))$ significantly smaller than the one obtained by quasi optimal single interval treatment. In general, larger values of $V(t_f)$ will produce smaller values of $U_\infty(\Rn(0),U(t_f),I(t_f),V(t_f))$. 
	This case includes the particular case where the treatment is interrupted at the very moment at which $U(t_f)=U^*$, but with $V(t_f) \not\approx 0$. This means that the critical value of $U$ needs to be reached as a steady state, not as a transitory one.
	\end{enumerate}
\end{theorem}
\begin{proof}
	The proof follows from the stability results shown in Sections~\ref{sec:equil_stab}, and \eqref{sec:InitCond}:
	\begin{enumerate}
	\item Given that $u_i=u^g$ is implemented for $t \in[t_i,t_f]$, $t_f$ is finite but large enough and $U_\infty (\Rn^g,U(t_i),I(t_i),V(t_i)) = U^*$, then $U(t_f)$ approaches $U^*$ and $V(t_f)$ approaches zero, from above, as $t_f$ increases. This means that at $t_f$, when the treatment is interrupted, $(U(t_f),I(t_f),V(t_f))$ is close to the unstable equilibrium set $\setX_s^{un}$. Then, by Property~\ref{propt:sinfty}.(ii), function $U_\infty(\Rn(0),U(t_f),I(t_f),V(t_f))$ is such that the closer $(U(t_f),I(t_f),V(t_f))$ is to the equilibrium point $(U^*,0,0)$, with $U(t_f)>U^*$, the closer will be $U_\infty(\Rn(0),U(t_f),I(t_f),V(t_f))$ to $U^*$, with $U_\infty<U^*$ (see the 'pine' shape of $U_\infty$ around $U^*$, for $I\approx 0$, in Figure~\ref{fig:Uinf_U})\footnote{Indeed, by the $\epsilon-\delta$ stability of the equilibrium state $(U^*,0,0)$, for each (arbitrary small) $\epsilon>0$, it there exists $\delta>0$, such that, if the system starts in a ball of radius $\delta$ centered at $(U^*,0,0)$, it will keeps indeterminately in the ball of radius $\epsilon$ centered at $(U^*,0,0)$. Furthermore, it is possible to define invariant sets around $(U^*,0,0)$ by considering the level sets of the Lyapunov function \eqref{ec:lya1}, with $\bar U=U^*$, or even the level sets of function $J(U,I,V):=U^*- U_\infty(\Rn,U,I,V)$, with a fixed $\Rn>0$. This way, once the system enters any arbitrary small level set of the latter functions, it cannot leaves the set anymore. See, Figure~\ref{fig:LevCurv}}.
	\item Given that $(U(t_f),I(t_f),V(t_f))$ approaches a steady state with $U(t_f) < U^*$, then $(U(t_f),I(t_f),V(t_f))$ is close to the stable equilibrium set $\setX_s^{st}$, when the treatment is interrupted. Then, the system will converge to an equilibrium with $U_\infty(\Rn(0),U(t_f),I(t_f),V(t_f))$ close to $U(t_f)$. Softer antiviral treatment  produces smaller values of $U(t_f)$ and, by Property~\ref{propt:sinfty}.(iii), smaller values of $U(t_f)$ produce smaller values of $U_\infty(\Rn(0),U(t_f),I(t_f),V(t_f))$. 
	\item Given that $(U(t_f),I(t_f),V(t_f))$ approaches a steady state with $U(t_f) > U^*$, then $(U(t_f),I(t_f),V(t_f))$ is close to the unstable equilibrium set, $\setX_s^{un}$, when the treatment is interrupted. Then, the system will converges to an equilibrium in the stable equilibrium set, $\setX_s^{st}$, with $U_\infty(\Rn(0),U(t_f),I(t_f),V(t_f)) < U^*$. Stronger antiviral treatment produces greater values of $U(t_f)$ and, by Property~\ref{propt:sinfty}.(ii), values of $U(t_f)$ farther from $U^*$, from above, produce values of $U_\infty(\Rn(0),U(t_f),I(t_f),V(t_f))$ farther from $U^*$, from below. When $U(t_f)$ is significantly greater than $U^*$, no matter how large is $t_f$ and how small is $V(t_f)$\footnote{Note that as long as $t_f$ is finite, $(U(t_f),I(t_f),V(t_f))$ cannot reach $\setX_s^{un}$, and so $V(t_f)$, even when arbitrary small, is greater than zero. So, once the social distancing is interrupted, the system evolves to an equilibrium in $\setX_s^{un}$.}, the system will evolve to an equilibrium in $\setX_s^{st}$, with $U_\infty(\Rn(0),U(t_f),I(t_f),V(t_f))$ significantly smaller than $U^*$. Furthermore, to go from $U(t_f)$ to $U_\infty(\Rn(0),U(t_f),I(t_f),V(t_f))$, for $t>t_f$, the system significantly increase $V(t)$, and this effect is known as a second outbreak wave.
	\item Given that $(U(t_f),I(t_f),V(t_f))$ is a transitory state, then it does not approach any equilibrium. This means that $V(t_f)$ is significantly greater than $0$, and according to Lemma~\ref{lem:Uinf_opt}, in Appendix~\hyperlink{sec:app2}{2}, the maximum of $U_\infty(\Rn(0),U(t_f),I(t_f),V(t_f))$ over 	$\Omega(\varepsilon) = \left\{(U,I,V) \in \mathbb R^3_{\ge0} : I\ge\varepsilon,V\ge\varepsilon\right\}$ is given by $- W(- \Rn U^* e^{-\Rn (U^* + \varepsilon + \frac{\delta}{p}\varepsilon)})/\Rn$, which is a decreasing function of $\varepsilon$, and reaches $U^*$ only when $\varepsilon=0$ (see Figure~\ref{fig:Uinf_U}). Then, independently of the value of $U(t_f)$, $U_\infty(\Rn(0),U(t_f),I(t_f),V(t_f))$ will be (maybe significantly) smaller than the one obtained with quasi optimal single interval treatment, in which $\varepsilon \approx 0$. %
	%\todo[inline]{Ale. Ac\'a se utiliza el Lemma final, con $V=\varepsilon$. Si se modifica algo de ese Lemma, hay que ver que esto quede acorde a esa modificaci\'on.}
\end{enumerate}
\end{proof}
\begin{figure}
	\centering
	\includegraphics[width=0.75\columnwidth]{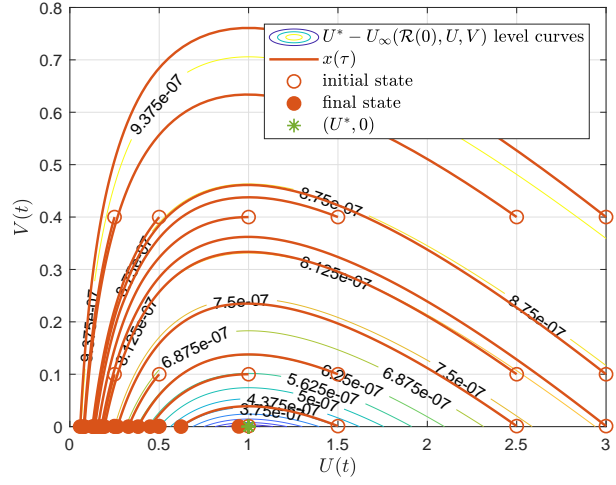}
	\caption{\small{Phase Portrait for system \eqref{eq:SysOrig} ($\beta = 1/2$, $\delta = 1/5$, $p = 2$, $c = 5$) in the $U,V$ plane ($I=(c/p)V$), for different starting points (red lines), and level curves of function $J(U,I,V):=U^*-U_\infty(\Rn,U,I,V)$, $I=(c/p)V$. Function $J$ is positive for all $(U,V) \neq (U^*,0)$, is null at $(U^*,0)$ and $\dot J(U,V)=0$ along the solution of system \eqref{eq:SysOrig} (since $U_\infty(\Rn,U,I,V)$ is). So its level sets are arbitrary small invariant sets around $(U^*,0)$. Note that starting states close to $(U^*,0)$ produce time evolution close to $U^*$, as determines the $\epsilon-\delta$ stability.}}
	\label{fig:LevCurv}
\end{figure}
%

%%%%%%%%%%%%%%%%%%%%%%%%%%%%%%%%%%%%%%%%%%
\subsection{Second control objective: minimizing the virus peak} \label{sec:seccont}
%%%%%%%%%%%%%%%%%%%%%%%%%%%%%%%%%%%%%%%%%%

The quasi optimal single interval treatment clearly accounts for a steady state condition, given that any realistic treatment needs to be interrupted at a finite time. Furthermore, given that only one antiviral dose, $u_i$, is considered for the treatment, once the quasi optimal single interval treatment is determined ($u_i=u^g$), also is the peak (maximum over $t$) of the virus, $\hat V$: i.e., there is a unique $\hat V$ for each single interval control action, $u_i$.

However, if a more general control action is considered, in such a way that $u_k$ assume several values in the interval from $t_i$ to $t_f$, $\hat V$ can be arbitrarily reduced. Indeed, given that $U_\infty$ depends only on the fact that $U\approx U^*$ and $V \approx 0$ at $t_f$, then stronger antiviral doses can be used at the beginning of the treatment to lower the peak of $V$. If for instance two consecutive single interval control actions are implemented ---the first one with a high dose, applied from $t_i$ to $t_1$, and the second one with the quasi optimal antiviral, $u^g(t_1)$, applied from $t_1$ to a large enough $t_f$--- a lower peak of $V$ will necessarily be obtained in contrast to one corresponding to the quasi optimal single interval control. 

Although this chapter is not devoted to analyze control strategies different from the single interval one, it is worth to remark this latter point since it states that: (1) both control objectives are independent, in the the sense that if a given upper bound for $V$ is stated from the the beginning (to avoid complication and/or to reduce the infectivity of the host) it is in general possible to design control strategies that both, make $V(t)$ not to overpasses the upper bound, and make $U_\infty \approx U^*$, and (2) the entire concept of a maximum or peak for $V$, for a given treatment, has sense only when $U_\infty \leq U^*$; since otherwise, a rebounds of the virus will occurs once the treatment is interrupted, and a new peak for $V$ may be reached.

Figures~\ref{fig:2step} and \ref{fig:2step2}, in the Simulation section, show an example of a two-steps interval treatment that produces a peak of $V$ smaller than the one corresponding to the quasi optimal single interval one.

%%%%%%%%%%%%%%%%%%%%%%%%%%%%%%%%%%%%%%%%%%%%%%%%%%%%%
\section{Simulation results}\label{sec:simres}
%%%%%%%%%%%%%%%%%%%%%%%%%%%%%%%%%%%%%%%%%%%%%%%%%%%%%

In this section each of the cases of Theorem \ref{teo:cont_sce}, together with the case of two-steps interval control action of Subsection \ref{sec:seccont} are simulated for data coming from patient 'A', introduced in section \ref{sec:IntPatA} and \ref{sec:simex}. As it was already said, $\delta_D = 2$ (days$^{-1}$), $EC_{50}^p = 75$ (mg) and the sampling time is selected to be $T=1$ day. Initial conditions are given by $(U_0,I_0,V_0)=(4\times 10^8,0,0.31)$. Also, recall that $U^*=5.44\times 10^7$ and the untreated peak of $V$ is given by $\hat V=1.98 \times 10^7$.

\subsection{Strong long-term treatment. Virus rebound}

Figure~\ref{fig:rebo} shows the time evolution of $U$ (logarithmic scale), $V$, and $u_k$ for patient 'A', when strong long-term antiviral treatment is implemented. The treatment starts at $t_i=4$ days and finished at $t_f=30$ days, while several strong doses are administered: $u_i=[21, 25, 35]$ mg. 

As it can be seen, at $t_f$ the value of $U$ is greater than $U^*$ while $V\approx 0$, so the viral load $V$ rebounds after some time, producing a second (and larger) peak. More important, $U_\infty$ ends up at a value significantly smaller than $U^*$. The values of $U_\infty$ and $\hat V$ corresponding to the three doses are given by $U_\infty = [ 7.98\times 10^6, 3.79\times 10^6, 1.07\times 10^5]$, and $\hat V = [4.84\times 10^6, 7.86\times 10^6, 1.34\times 10^7]$, respectively.

To have a better idea of how the system behaves around state $(U^*,0)$, Figure~\ref{fig:rebo2} shows the phase portrait in the space $U,V$, together with the level curves of the Lyapunov function $J(U,I,V):=U^*-U_\infty(\Rn(0),U,I,V)$. At time $t_f$, when the treatment is interrupted, $V(t_f) \approx 0$ and $U(t_f) > U^* $, so the system is close to an unstable equilibrium point. So, for $t>t_f$ the state is attracted to an equilibrium in the AS equilibrium set $\setX_s^{st}$, following outer level curves of $J$. Outer level curves of $J$ means both, a small $U_\infty$ and a large $\hat V$.
\begin{figure}
	\centering
	\begin{subfigure}{.5\textwidth}
		\includegraphics[width=1\textwidth]{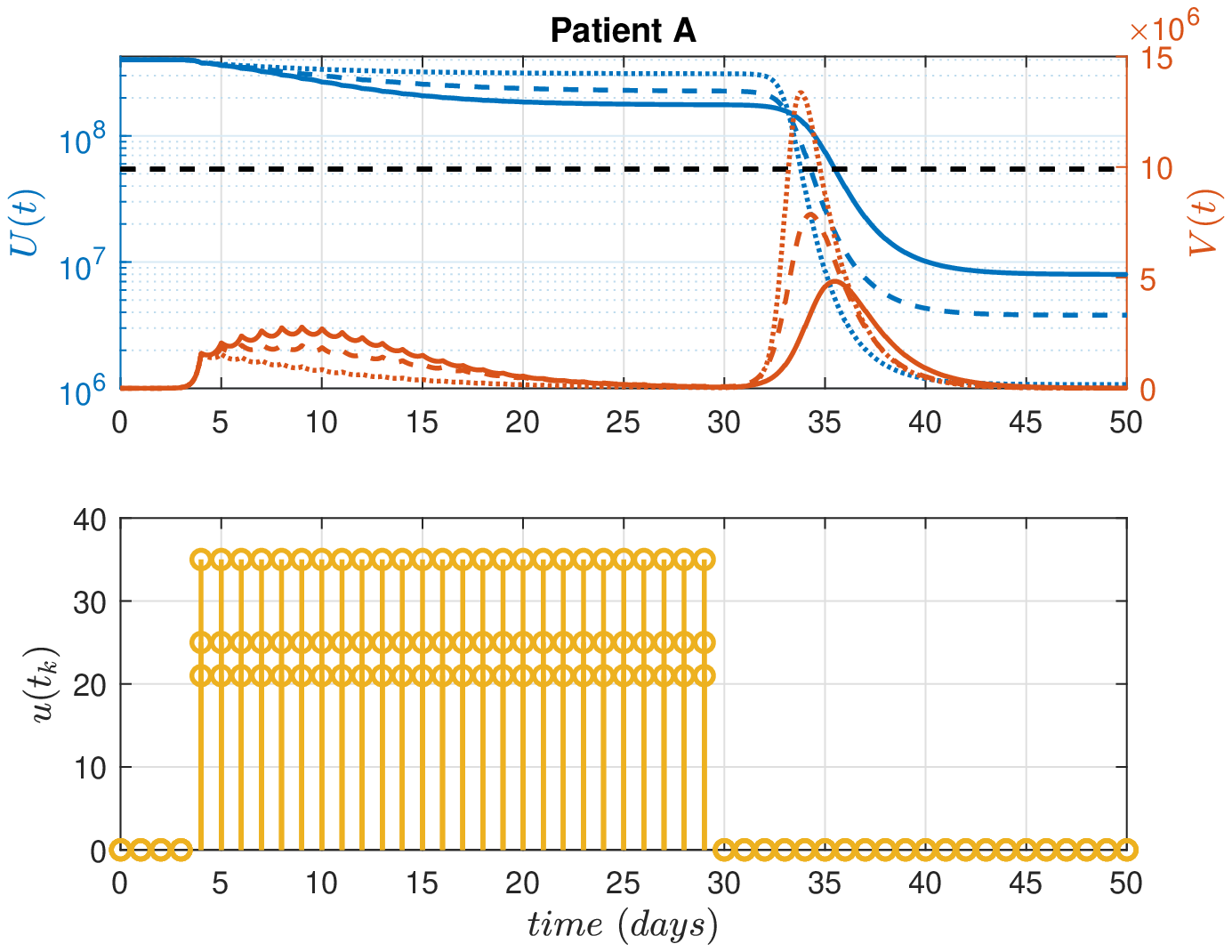}
	\caption{}
	\label{fig:rebo}
	\end{subfigure}%
	\hspace*{0.2truecm}
	\begin{subfigure}{.5\textwidth}
			\centering
	\includegraphics[width=1\textwidth]{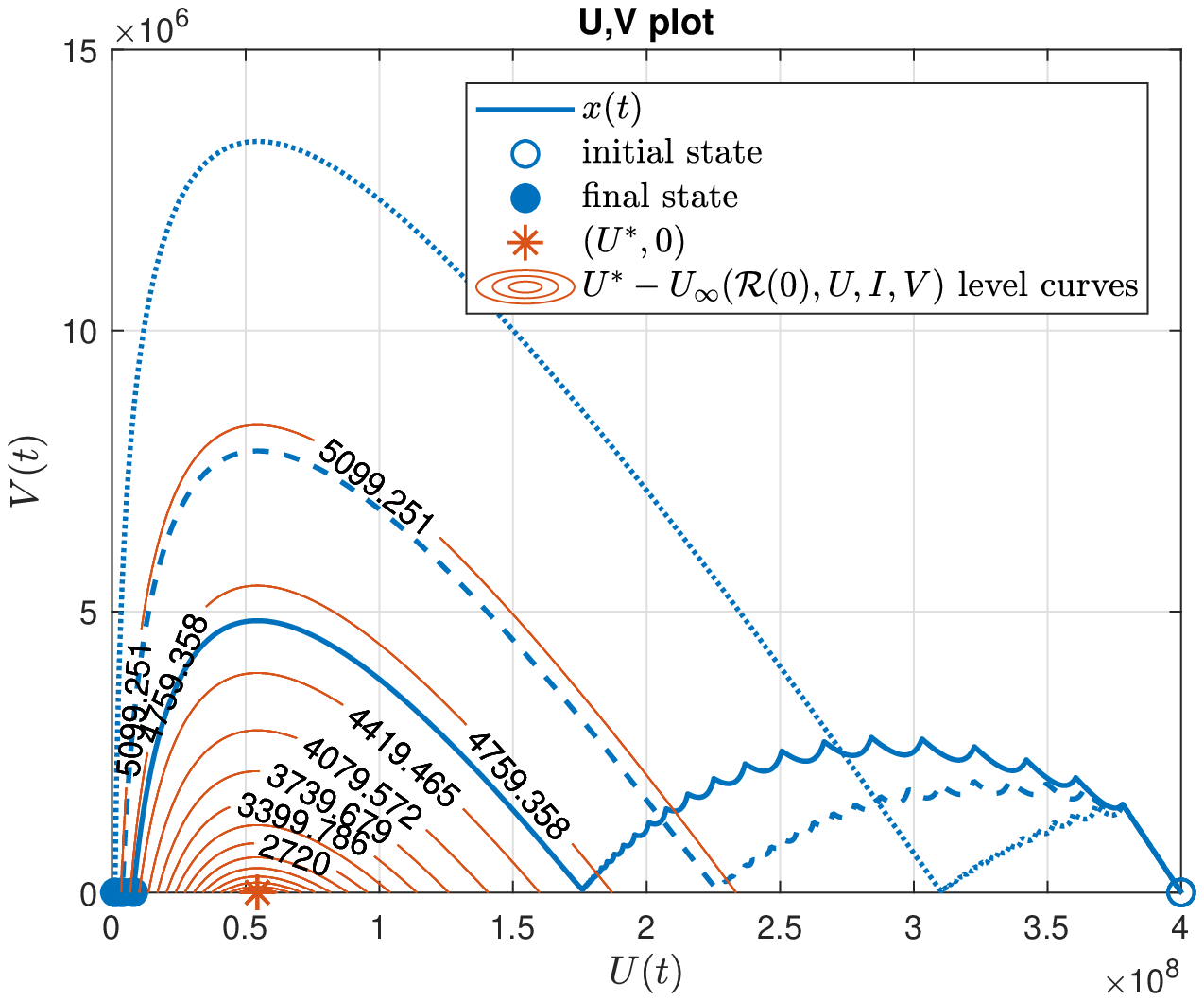}
	\caption{}
	\label{fig:rebo2}
	\end{subfigure}
		\caption{\small{(a) Time evolution of virtual patient 'A', with different doses of antiviral:  $u_i=21$ mg, solid line, $u_i=25$ mg, dashed line, $u_i=35$ mg, dotted line. (b) Phase portrait in the $U,V$ space, and level curves of the Lyapunov function $J(U,I,V):=U^*-U_\infty(\Rn(0),U,I,V)$, around $(U^*,0)$.}}
\end{figure}

\subsection{Soft long-term treatment}

Figure~\ref{fig:subop} shows the time evolution of $U$ (logarithmic scale), $V$, and $u_k$, when soft long term antiviral treatment is implemented. The treatment starts at $t_i=4$ days and finished at $t_f=30$ days, while several soft doses are administered: $u_i=[4, 6, 8]$ mg. 

As it can be seen, at $t_f$ the value of $U$ is smaller than $U^*$, while $V\approx 0$, so the viral load $V$ decreases after the treatment is interrupted. The values of $U_\infty$ and $\hat V$ corresponding to the three doses are given by $U_\infty = [ 8.98\times 10^6, 1.90\times 10^7, 3.25\times 10^7]$, and $\hat V = [1.36\times 10^7, 1.16\times 10^7, 9.70e\times 10^6]$.

Figure~\ref{fig:subop2} shows the phase portrait in the space $U,V$, together with the level curves of the Lyapunov function $J(U,I,V)$. At time $t_f$, when the treatment is interrupted, $V(t_f) \approx 0$ and $U(t_f) < U^* $, so the system is close to a stable equilibrium point. So, for $t>t_f$ the state remains almost unmodified.

\begin{figure}
	\centering
	\begin{subfigure}{.5\textwidth}
		\includegraphics[width=1\textwidth]{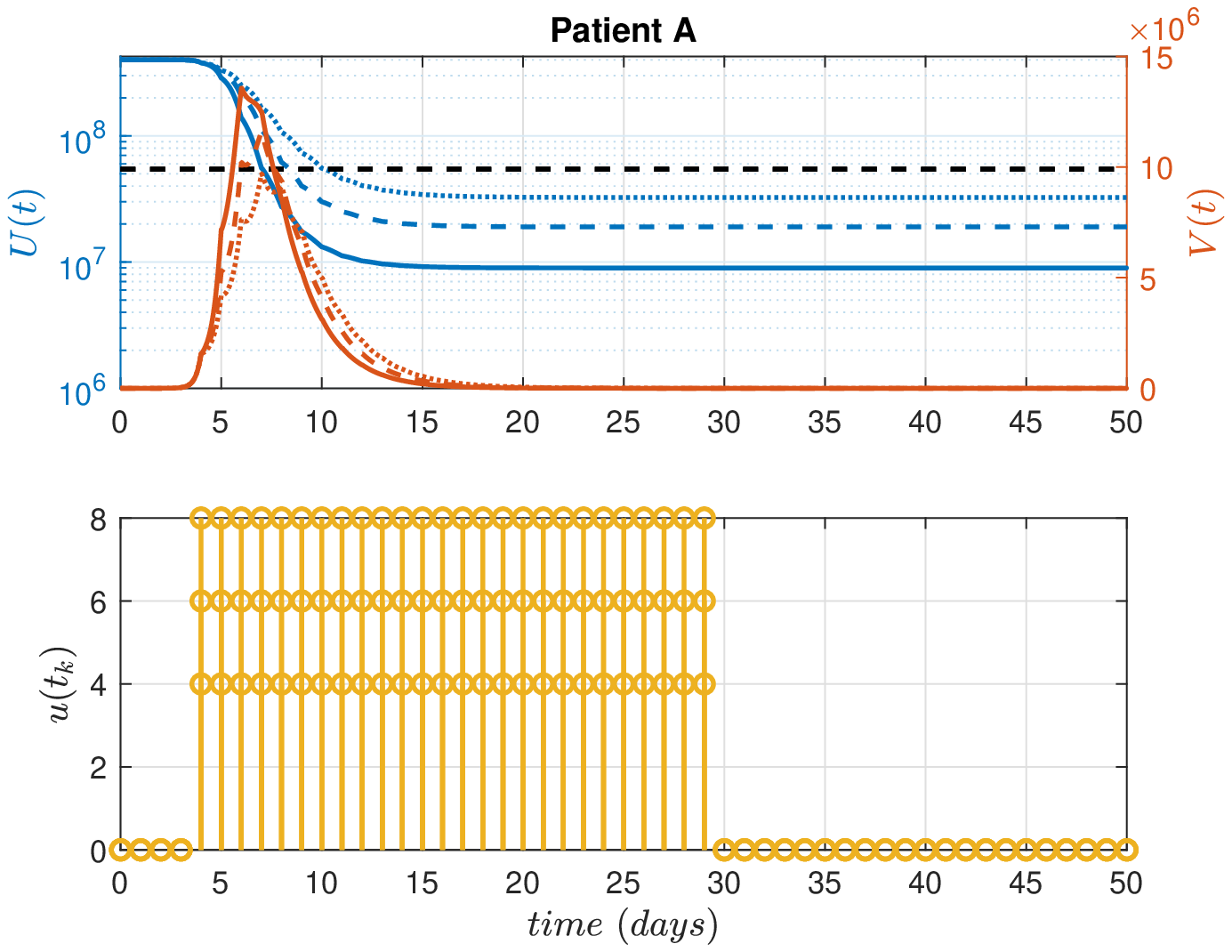}
	\caption{}
	\label{fig:subop}
	\end{subfigure}%
	\hspace*{0.2truecm}
	\begin{subfigure}{.5\textwidth}
	\centering
	\includegraphics[width=1\textwidth]{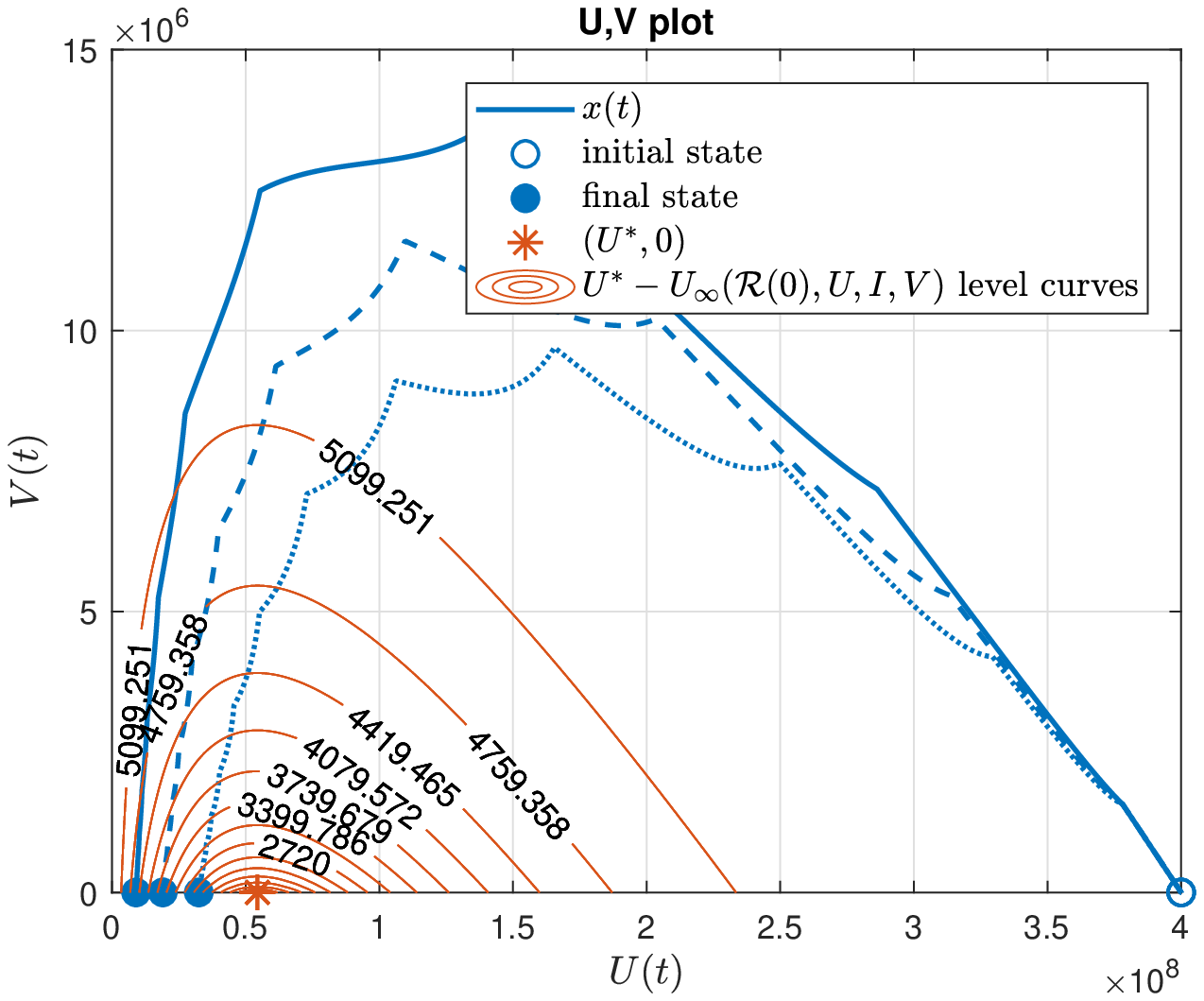}
	\caption{}
	\label{fig:subop2}
	\end{subfigure}
	\caption{\small{(a) Time evolution of virtual patient 'A', with different doses of antiviral:  $u_i=4$ mg, solid line, $u_i=6$ mg, dashed line, $u_i=8$ mg, dotted line. (b) Phase portrait in the $U,V$ space, and level curves of the Lyapunov function $J(U,I,V):=U^*-U_\infty(\Rn(0),U,I,V)$, around $(U^*,0)$.}}
\end{figure}

\subsection{Quasi optimal single interval treatment} \label{sec:optsimul}

Figure~\ref{fig:qopt} shows the time evolution of $U$ (logarithmic scale), $V$, and $u_k$, when the quasi optimal single interval antiviral treatment is administered. The treatment starts at $t_i=4$ days and finished at $t_f=30$ days, while the Goldilocks dose is given by $u_i=u^g(t_i)=10.5$ mg. The values of $U_\infty$ and $\hat V$ are given by $U_\infty =5.34\times 10^7$ and $\hat V = 7.73\times 10^6$, respectively

Figure~\ref{fig:qopt2} shows the phase portrait in the space $U,V$, together with the level curves of the Lyapunov function $J(U,I,V)$. As it can be seen, the system follows the only one trajectory that goes directly from $(U_0,V_0)$ to $(U^*,0)$: any other path goes necessarily to an equilibrium with $U_\infty <U^*$. 
\begin{figure}
	\centering
	\begin{subfigure}{.5\textwidth}
		\includegraphics[width=1\textwidth]{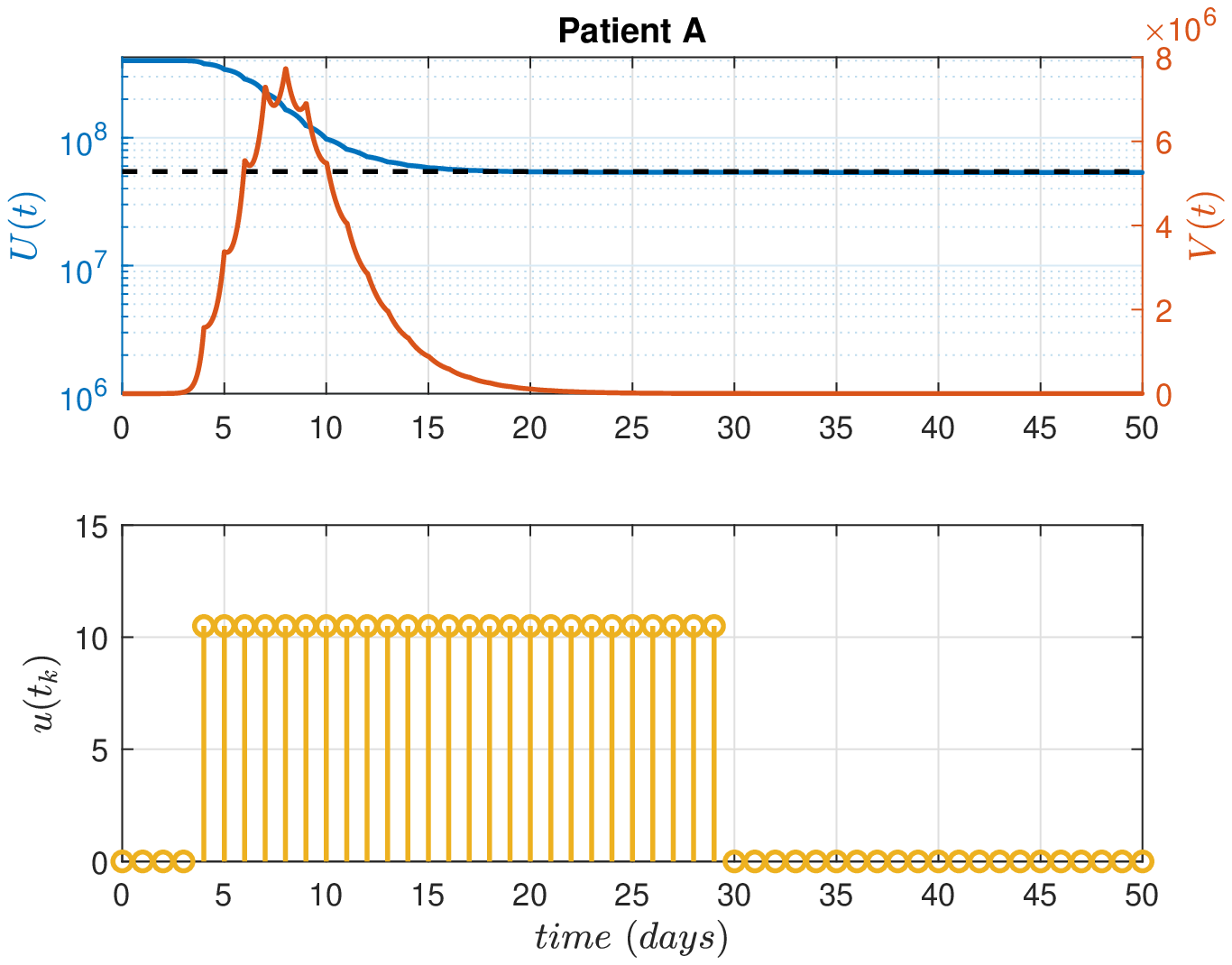}
	\caption{}
	\label{fig:qopt}
	\end{subfigure}%
	\hspace*{0.2truecm}
	\begin{subfigure}{.5\textwidth}
			\centering
	\includegraphics[width=1\textwidth]{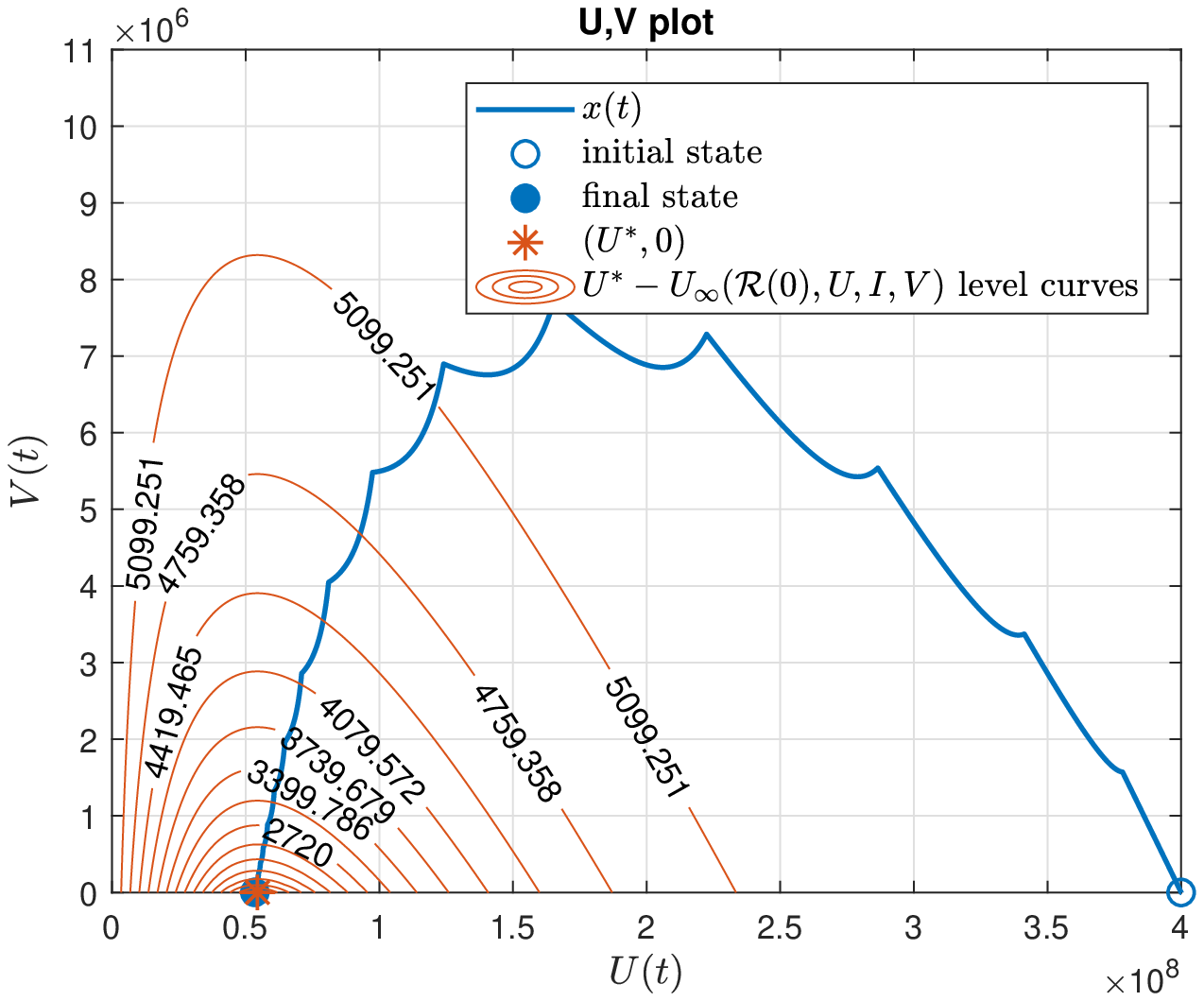}
	\caption{}
	\label{fig:qopt2}
	\end{subfigure}
	\caption{\small{(a) Time evolution of virtual patient 'A', with $u_i=u^g=10.5$ mg of antivirals. (b) Phase portrait in the $U,V$ space, and level curves of the Lyapunov function $J(U,I,V):=U^*-U_\infty(\Rn(0),U,I,V)$, around $(U^*,0)$.}}
\end{figure}

\subsection{Short-term treatment}

Figure~\ref{fig:short} shows the time evolution of $U$ (logarithmic scale), $V$, and $u_k$, when a short term treatment is implemented. The treatment starts at $t_i=4$ days and finished at $t_f=15$ days, while several doses - smaller and greater than $u^g(t_i)$ are administered: $u_i=[10, 15, 20, 25]$ mg. The values of $U_\infty$ and $\hat V$ corresponding to the four doses are given by $U_\infty = [2.98\times 10^67, 1.39\times 10^7, 5.06\times 10^7, 2.26\times 10^7]$, and $\hat V = [8.05\times 10^6, 4.98\times 10^6, 6.67\times 10^6, 1.01\times 10^1]$.

Figure~\ref{fig:short2} shows the phase portrait in the space $U,V$. Given that trajectories go along the level curves of the Lyapunov function $J(U,I,V)$, any short term treatment - i.e., producing $V(t_f) \not\approx 0$ - will make the system to surround the state $(U^*,0)$ by an outer level curve, thus finishing at some $U_\infty$ significantly smaller than $U^*$. As before, outer level curves of $J$ means both, a small $U_\infty$ and a large $\hat V$.

\begin{figure}
	\centering
	\begin{subfigure}{.5\textwidth}
		\includegraphics[width=1\textwidth]{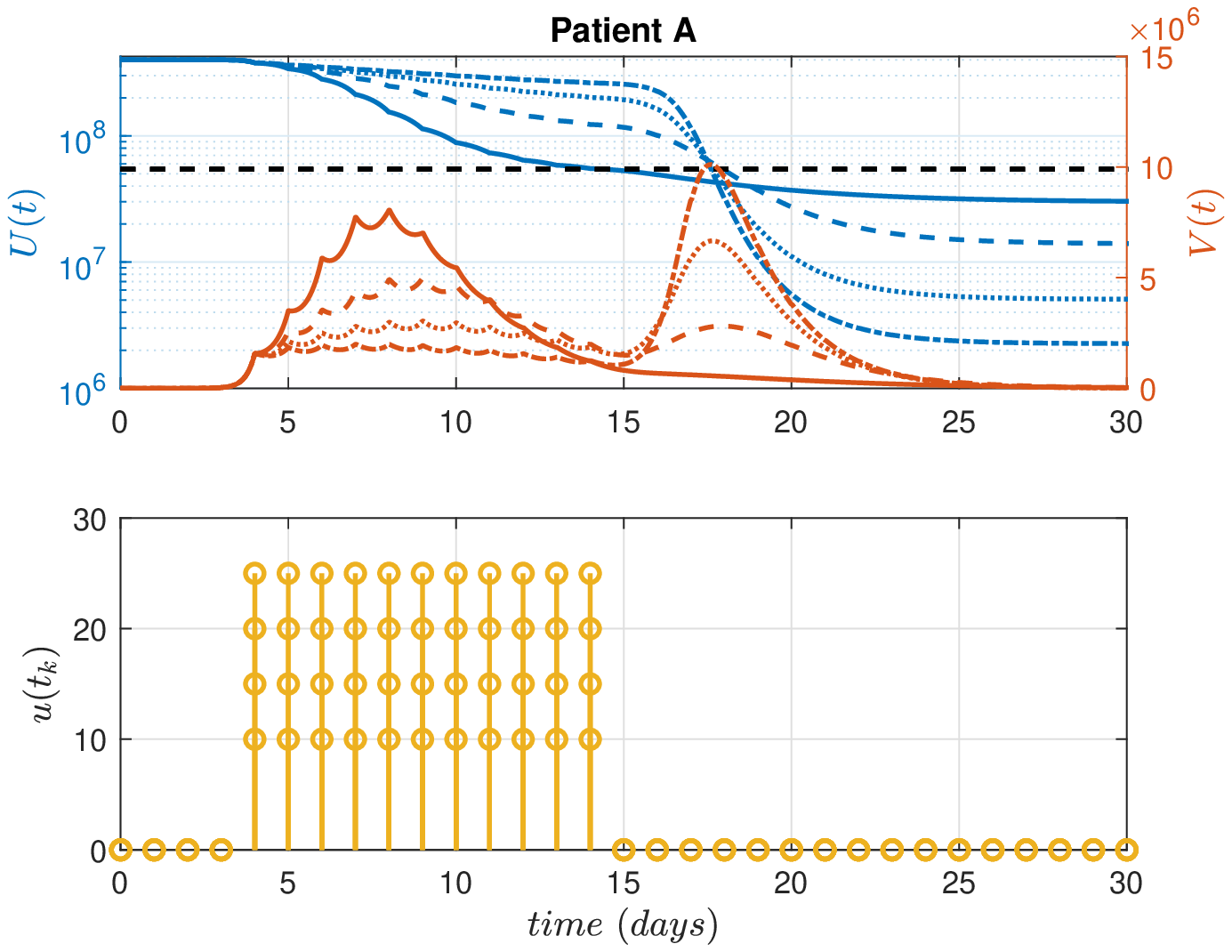}
	\caption{}
	\label{fig:short}
	\end{subfigure}%
	\hspace*{0.2truecm}
	\begin{subfigure}{.5\textwidth}
			\centering
	\includegraphics[width=1\textwidth]{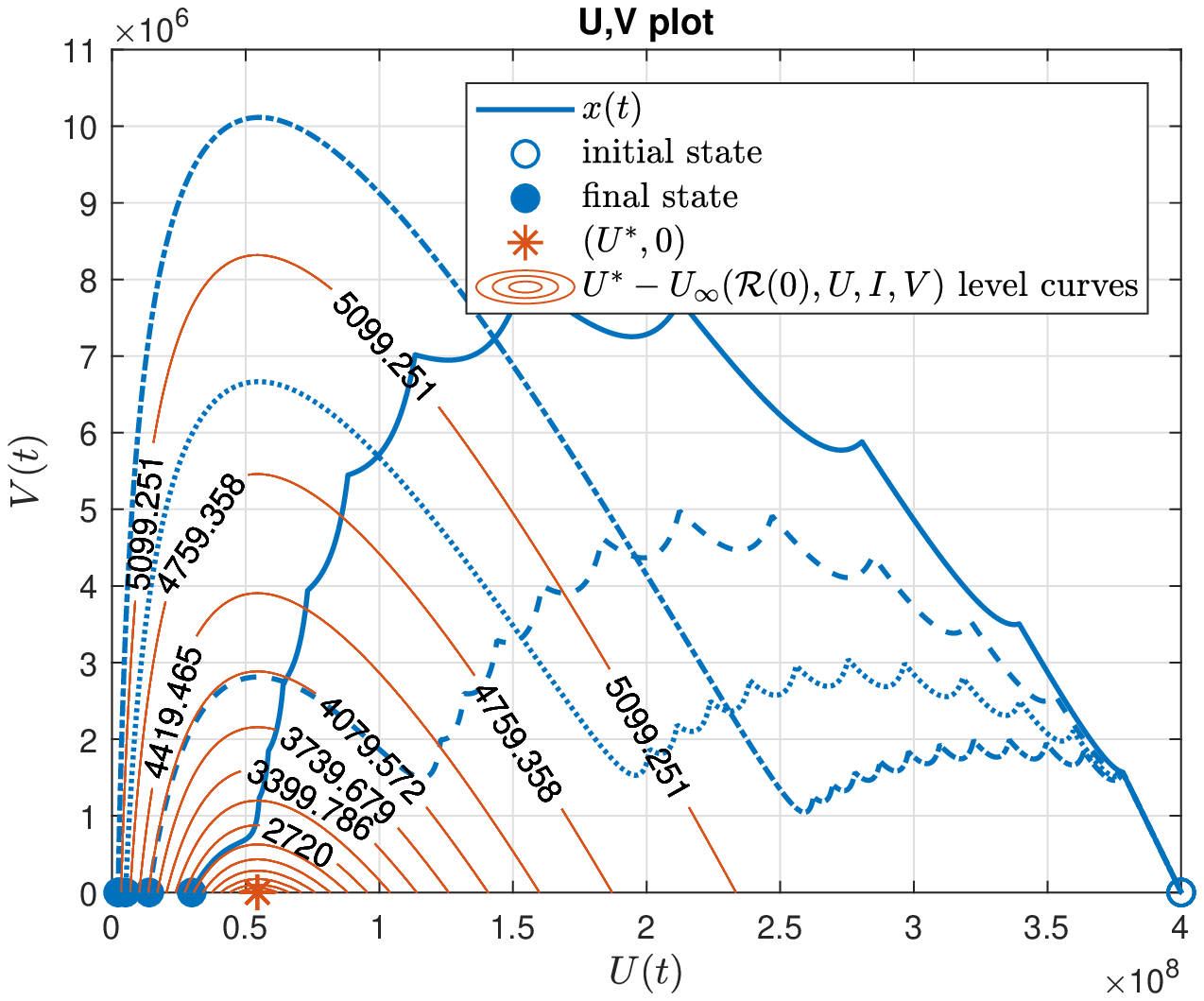}
	\caption{}
	\label{fig:short2}
	\end{subfigure}
	\caption{\small{(a) Time evolution of virtual patient 'A', with different doses of antiviral: $u_i=10$ mg, solid line, $u_i=15$ mg, dashed line, $u_i=20$ mg, dotted line, and $u_i=25$ mg, dashed-dotted line. (b) Phase portrait in the $U,V$ space, and level curves of the Lyapunov function $J(U,I,V):=U^*-U_\infty(\Rn(0),U,I,V)$, around $(U^*,0)$.}}
\end{figure}

\subsection{Two-steps treatment, lowering the peak of $V$}

Finally, a scenario is simulated to show that always it is possible to lower the the peak of $V$ - while maintaining $U_\infty \approx U^*$ - if a control sequence more complex that the single interval one is implemented. Figure \ref{fig:2step} shows the time evolution of $U$ (solid blue line, logarithmic scale) and $V$ (solid red line) corresponding to a two-steps interval control: the first step consisting in $u_i = 25$ mg, from $t_i=4$ to $t_m=30$ days, and the second one consisting in $u_i=u^g(t_m)=5.6$ mg, from $t_m=30$ to $t_f=60$ (solid line). Also, the quasi optimal single interval control of Subsection \ref{sec:optsimul} is shown, to compare the performance (dashed line). As it can be seen, the peak of $V$ is significantly reduced: from $\hat V = 7.73\times 10^6$ to $\hat V = 2.57 \times 10^6$, while $U_\infty$ is almost the same in both cases. Figure \ref{fig:2step2} shows the phase portraits of the two control strategies (solid line, two-steps control; dashed line, single interval control), where it can be seen also the reduction of the virus peak. This simple two-step strategy shows that with a more sophisticated control strategy (i.e., by means of a proper optimal control formulation) the virus peak can be arbitrarily reduced, maintaining the condition $U_\infty \approx U^*$. This is indeed, matter of future research.
\begin{figure}
	\centering
	\begin{subfigure}{.5\textwidth}
		\includegraphics[width=1\textwidth]{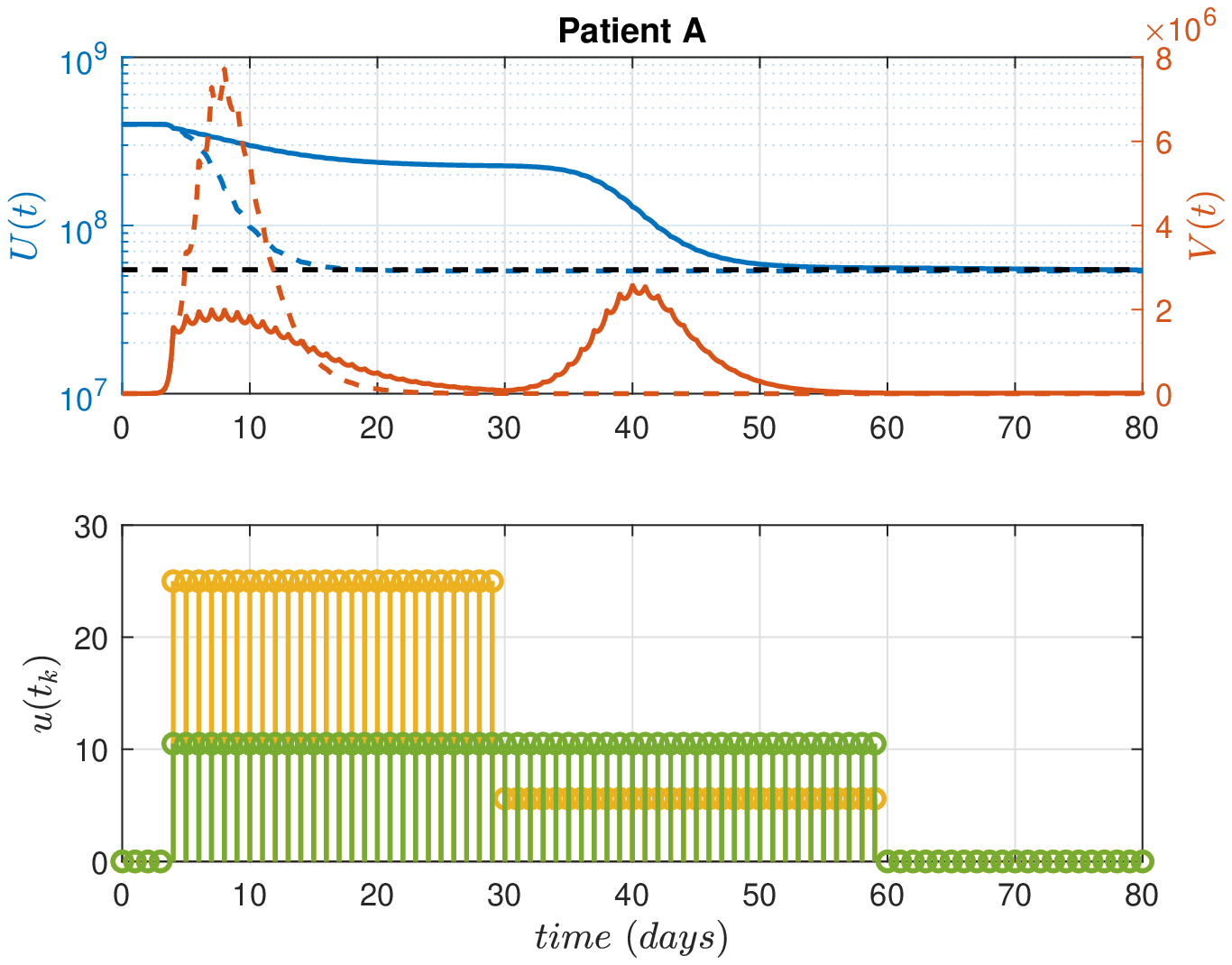}
	\caption{}
	\label{fig:2step}
	\end{subfigure}%
	\hspace*{0.2truecm}
	\begin{subfigure}{.5\textwidth}
			\centering
	\includegraphics[width=1\textwidth]{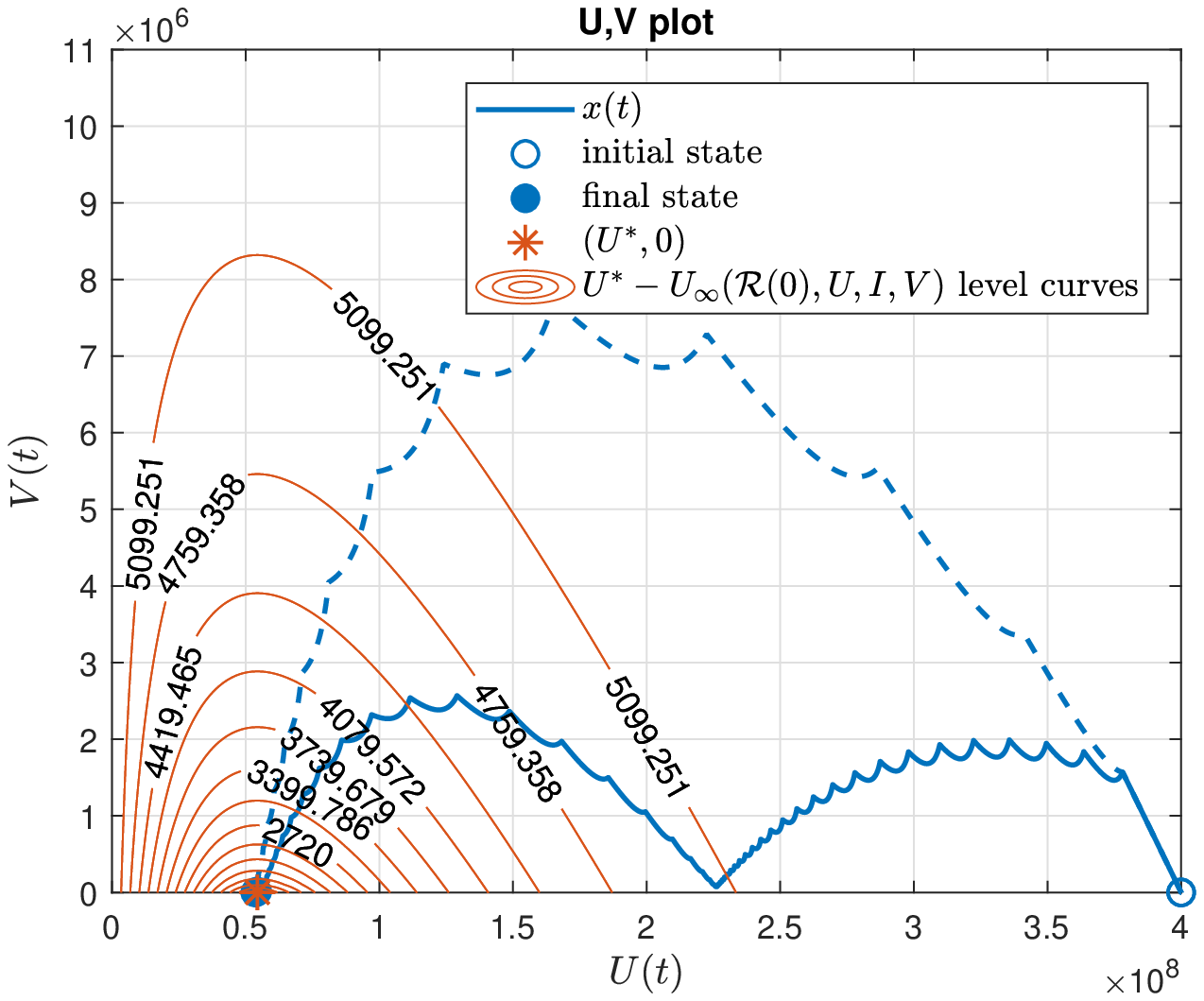}
	\caption{}
	\label{fig:2step2}
	\end{subfigure}
	\caption{\small{(a) Time evolution of virtual patient 'A', with different doses of antiviral: $u_i=25$ mg, from $t_i=4$ to $t_m=30$ days, and $u_i=u^g(t_m)=5.6$ mg, from $t_m=30$ to $t_f=60$ (solid line). In dashed line is plotted the quasi optimal single interval treatment. (b) Phase portrait in the $U,V$ space.}}
\end{figure}

%%%%%%%%%%%%%%%%%%%%%%%%%%%%%%%%%%%%%%%%%%%%%%%%%%%%%
\section{Conclusions and future works}\label{sec:conc}
%%%%%%%%%%%%%%%%%%%%%%%%%%%%%%%%%%%%%%%%%%%%%%%%%%%%%

In this work, the stability and general long term behavior of UIV-type models have been fully analysed. A quasi optimal control action - consisting in the finite-time single interval antiviral treatment producing the minimal possible final amount of death cells - was found. The analysis shows also that more complex control strategies can account for both control objectives simultaneously: minimize the virus peak, while keeping the final amount of death cells at its maximum.
A detailed analysis of subotimal scenarios permits to enumerate the following main results:
\begin{enumerate}
    \item To apply soft antiviral treatment during a long time (even no treatment at all), expecting the non-infected cells would evolve alone to the critical value $U^*$, is not an option. Open loop $U_\infty$ is in general significantly smaller than $U^*$ (particularly for the reported values of $\Rn$ for the COVID-19).
    \item To apply strong antiviral treatment for a long time, expecting the virus will die out alone is not an option. Strong antiviral produces an values of $U$ at the end of the treatment larger than $U^*$, but this final values are artificially stable steady states, since once the treatment is interrupted or reduced, a virus rebound will necessarily occurs at some future time, and $U_\infty$ will be significantly smaller than $U^*$.
    \item To apply any antiviral treatment (soft or strong) for short period of time, such that the system is not able to reach a quasi steady state (i.e., when $V$ at the end of the treatment is not close to zero) is not an option. If the treatment is interrupted at a transient state, the initial conditions for the next time period are such that $U_\infty$ will be significantly smaller than $U^*$.
    \item According to the latter results, the best option is to apply an antiviral treatment such that the system reaches a quasi steady state with $U\approx U^*$ and $V\approx 0$ at the end of the treatment. This is what we call "the quasi optimal single interval antiviral treatment", since it makes the system to approach the maximal final value of uninfected cells ($U_\infty \approx U^*$), without infection rebounds.
    \item An important point to be remarked is that the quasi optimal single interval antiviral treatment does not determine the peak of the virus. Quasi optimal conditions for $U_\infty$ are stationary, while condition for minimizing $\hat V$ are transitory, so both objective can be accounted for simultaneously.
\end{enumerate}

Future works include the study of more complex control strategies (mainly model based control strategies as MPC and similar) and the explicit consideration of time-varying immune system. 

%%%%%%%%%%%%%%%%%%%%%%%%%%%%%%%%%%%%%%%%%%%%%%%%%%%%%
\section*{Appendix \hypertarget{sec:app1}{1}: Stability theory}
%%%%%%%%%%%%%%%%%%%%%%%%%%%%%%%%%%%%%%%%%%%%%%%%%%%%%%

In this section some basic definitions and results are given concerning the asymptotic stability of sets and Lyapunov theory,
in the context of non-linear continuous-time systems (\cite{rawlings2017model}, Appendix B). All the following definitions are referred to system 
\begin{eqnarray}\label{eq:difeqini}
\dot x (t) = f(x(t)),~~x(0)=x_0,
\end{eqnarray}
where $x$ is the system state constrained to be in $\setX \subseteq \R^n$, $f$ is a Lipschitz continuous nonlinear function,
and $\phi(t;x)$ is the solution for time $t$ and initial condition $x$.
\begin{definition}[Equilibrium set]
	Consider system~\ref{eq:difeqini} constrained by $\setX$. The set $\setX_s \subset \setX$ is an equilibrium set if each point $x \in \setX_s$ is such that $f(x) =  0$ (this implying that $\phi(t;x)  =  x$ for all $t \geq 0$).
\end{definition}
\begin{definition}[Attractivity of an equilibrium set]\label{def:attrac_set}
	Consider system~\ref{eq:difeqini} constrained by $\setX$ and a set $\setY \subseteq \setX$. A closed equilibrium set $\setX_s \subset \setX$ is attractive in $\setY$ if $\lim_{t \rightarrow \infty} \|\phi(t;x)\|_{\setX_s} =0$ for all $x \in \setY$. If  $\setY$ is a $\varepsilon$-neighborhood of $\setX_s$ for some $\eta>0$, we say that $\setX_s$ is locally attractive.
\end{definition}
We define the \textit{domain of attraction} (DOA) of an attractive set $\setX_s$ for the system~\ref{eq:difeqini} to be the set of all initial states $x$ such that $\|\phi(t;x)\|_{\setX_s} \to 0$ as $t\to\infty$. We use the term \textit{region of attraction} to denote any set of initial states contained in the domain of attraction.

A closed subset of an attractive set (for instance, a single equilibrium point) is not necessarily attractive. On the other hand, any set containing an attractive set is attractive, so the significant attractivity concept in a constrained system is given by the smallest one\footnote{Given two different attractive sets in $\setX$ with the same DOA, one must be contained in the other. So the family of all attractive sets in $\setX$ with the same DOA is a totally ordered set under the set inclusion (nested family). An arbitrary (finite, countable, or uncountable) intersection of nested nonempty closed subsets of a compact space is a nonempty compact set~\cite{kelley2017general}. Then if one element of the family is bounded, and therefore compact, the intersection of all the family is a nonempty compact set. This set is the \textit{smallest atractive set}.}.

\begin{definition}[Local $\epsilon - \delta$ stability of an equilibrium set]\label{def:eps_del_stab}
	Consider system~\ref{eq:difeqini} constrained by $\setX$. A closed equilibrium set $\setX_s \subset \setX$ is $\epsilon-\delta$ locally stable if for all $\epsilon >0$ there exists $\delta>0$ such that if  $\|x\|_{\setX_s}<\delta$ then $\|\phi(t;x)\|_{\setX_s} < \epsilon$, for all $t \geq 0$.
\end{definition}

Unlike attractive sets, a set containing a locally $\epsilon-\delta$ stable equilibrium set is not necessarily locally $\epsilon-\delta$ stable. Even more, a closed subset of a locally $\epsilon-\delta$ stable equilibrium set (for instance, a single equilibrium point) is not necessarily locally $\epsilon-\delta$  stable. However, any (finite) union of equilibrium sets locally $\epsilon-\delta$ stable is also locally $\epsilon-\delta$ stable. So the significant stability concept in a constrained system is given by the largest one.

Although a finite union of equilibrium set locally $\epsilon-\delta$ stable is also locally $\epsilon-\delta$ stable, in general we cannot extend this result to the case of arbitrary unions of points. Thus, even when every equilibrium point of an equilibrium set is locally $\epsilon-\delta$ stable, we cannot assure that the whole set would be locally $\epsilon-\delta$ stable. This is due to the fact that given a fixed $\epsilon$ the $\delta$ chosen for each point depend on the point and so the infimum of them could be zero.  However, if in addition we also assume that the set is compact, then the stability of the set can be inherited from the stability of its points.

\begin{lemma}\label{lem:stab}
     Let $\setX_s$ be a compact equilibrium set. If every $x_s\in \setX_s$ is $\epsilon-\delta$ locally stable, then $\setX_s$ is $\epsilon-\delta$ locally stable.
\end{lemma}
\begin{proof}
Given $\epsilon>0$, there exists  $\delta=\delta(x_s)>0$ for each $x_s\in\setX_s$ such that if $x\in B_{\delta(x_s)}(x_s)$ then $\phi(t;x)\in B_\epsilon(x_s)$ for $t\ge 0$. The family of $\delta$-balls form a open cover of $\setX_s$. Let us denote the union of this cover $V$, i.e. $V:=\bigcup\{B_{\delta(x_s)}(x_s): x_s\in\setX_s\}$. Since $\setX_s$ is compact and the complement of $V$ is closed, then the distance between them is strictly positive, i.e. $\delta^*:=d(\setX_s,V^c)>0$. Therefore, the $\delta^*$ neighborhood of the equilibrium set $\setX_s$ is contained in $V$. Thus if $x\in B_{\delta^*}(\setX_s)\subset V$ then $\phi(t;x)\in B_\epsilon(\setX_s)$ for $t\ge 0$ .Therefore $\setX_s$ is $\epsilon-\delta$ locally stable.
\end{proof}

\begin{definition}[Asymptotic stability (AS) of an equilibrium set]\label{def:AS}
	Consider system~\ref{eq:difeqini} constrained by $\setX$ and a set $\setY \subseteq \setX$. A closed equilibrium set $\setX_s \subset \setX$ is asymptotically stable (AS) in $\setY$ if it is $\epsilon-\delta$ locally stable and attractive in $\setY$.
\end{definition}

Next, the theorem of Lyapunov, which refers to single equilibrium points and provides sufficient conditions for both, local $\epsilon-\delta$  stability and assymptotic stability, is introduced.
\begin{theorem}\label{theo:lyap}\emph{(Lyapunov's stablity theorem \cite[Theorem~4.1]{khalil2002nonlinear})}
	Consider system~\ref{eq:difeqini} constrained by $\setX$ and an equilibrium state $x_s \in \setX_s$. Let $\setY\subset\setX$ be a neighborhood of $x_s$ and consider a	function $V(x): \setY \rightarrow\R$ such that $V(x)>0$ for $x \neq x_s$, $V(x_s)=0$ and $\dot V(x(t)) \leq 0$, denoted as Lyapunov function.
	Then, the existence of such a function in a neighborhood of $x_s$ implies that $x_s \in \setX_s$ is locally $\epsilon-\delta$ stable in $\setY$. If in addition
	$\dot{V}(x(t)) < 0$ for all $x \neq x_s$, then $x_s$ is asymptotically stable in $\setY$.
\end{theorem}
%

%%%%%%%%%%%%%%%%%%%%%%%%%%%%%%%%%%%%%
\section*{Appendix \hypertarget{sec:app2}{2}: Maximum of $U_\infty$}
%%%%%%%%%%%%%%%%%%%%%%%%%%%%%%%%%%%%%%%%%%%%%%%%%%%%%

As mentioned previously, $U_\infty$ can be expressed as a function of $U$, $I$ and $V$ as follows
\begin{eqnarray} \label{eq:Sinf1}
	U_\infty(U,I,V)= -\frac{W(-\Rn Ue^{-\Rn(U + I +\frac\delta p V)})}{\Rn} 
\end{eqnarray}
with $\Rn$, $\delta$ and $p$ fixed. For each $\varepsilon\ge0$ let us define a domain of $\setX$ given by  
\begin{eqnarray} \label{eq:Omega.ep}
	\Omega(\varepsilon) = \left\{(U,I,V) \in \setX:  I\ge\varepsilon,V\ge\varepsilon\right\}.
\end{eqnarray}
The following Lemma describe the behavior of the maximum of $U_\infty$ on each $\Omega(\varepsilon)$.

\begin{lemma}[Maximum of the function $U_\infty$]\label{lem:Uinf_opt}
	Consider the function $U_\infty$ given by~\eqref{eq:Sinf1} and for each $\varepsilon\ge0$ the domains $\Omega(\varepsilon)$ given by~\eqref{eq:Omega.ep}.  Then the maximum of $U_\infty(U,I,V)$ in $\Omega(\varepsilon)$ is reached in $(U^*,\varepsilon,\varepsilon)$.
    In particular, the maximum value of $U_\infty$ over $\Omega(0)$ is reached in $(U^*,0,0)$ and is given by $U_\infty(U^*,0,0) = U^*$, where $U^* = 1/\Rn$.
\end{lemma}

\begin{proof}
According to \eqref{eq:Sinf1}, $U_\infty$ can be written as
\[	U_\infty(U,I,V)= -\frac{W(-f(U,I,V))}{\Rn}, \]
with $f(U,I,V)=\Rn U e^{-\Rn(U + I+\delta/pV)}$. Since $-W(-\cdot)$ is an increasing (injective) function then $U_\infty(U,I,V)$ achieves its maximum over $\Omega(\varepsilon)$ at the same values as $f(U,I,V)$. Then, we focus our attention in finding the maximum (and the maximizing variables) of $f(U,I,V)$.
    
Through the change of variables $x = \Rn U$ and $y = \Rn(I+\delta/pV)$, $f$ can be studied as a function of the form $g(x,y) = xe^{-(x+y)}$.
Note that $(U,I,V)\in\Omega(\varepsilon)$ if and only if $x\ge0$ and $y\ge\eta$  where $\eta:=\Rn(1+\tfrac{\delta}{p})\varepsilon\ge0$. 
Therefore to find extremes of $f$ in $\Omega(\varepsilon)$ it is enough to study the extreme points of $g$ over $\Omega'=\{(x,y)\in\R^2_{\ge0}:~y\ge\eta\}$.
    
Since $\nabla g=[(1-x) e^{-(x+y)} , -xe^{-(x+y)}]$
does not vanish and $g\to0$ when $\|(x,y)\|\to\infty$, then the maximum is reached at the boundaries of $\Omega'$. 
A simple analysis shows that $g$ restricted to the boundary of $\Omega'$ achieves its maximum in $(1,\eta)$.
This means that $f(U,I,V)$ achieves its maximum in $U = 1/\Rn =U^*$ and $I =V=\varepsilon$.

In particular, when $\varepsilon = 0$,  $f(U,I,V)$ reaches its maximum in $(U^*,0,0)$. Furthermore, $$U_\infty(U^*,0,0)  = -\frac{W(-f(U^*,0,0))}{\Rn} = -\frac{W(-1/e)}{\Rn} = \frac1\Rn = U^*,$$ 
which concludes the proof.
\end{proof}

\begin{figure}
	\centering
	\includegraphics[width=0.75\columnwidth]{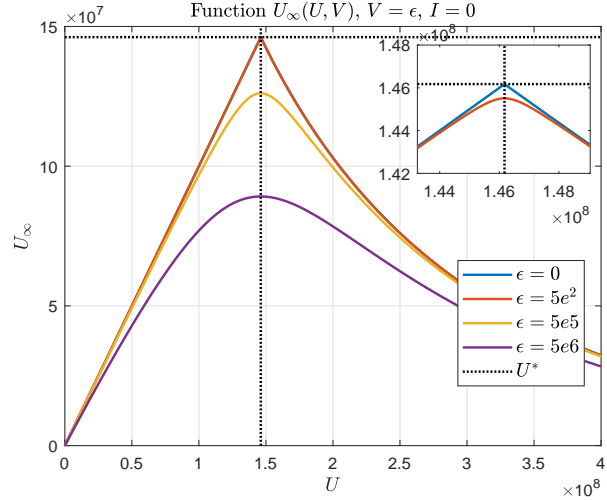}
	\caption{\small{Function $U_\infty(U,V)$, for different values of $\epsilon$, when $U \in [0,U_{\max}]$, $V = \epsilon$ and $I = 0$. As it can be seen, the supremum of $U_\infty$ (given by $U^*$) is achieved when $U=U^*$ and $I\rightarrow 0$.}}
	\label{fig:Uinf_U}
\end{figure}

%----------------------------------------------------------------------------------------
%	BIBLIOGRAPHY
%----------------------------------------------------------------------------------------

\bibliographystyle{elsarticle-num} 
\bibliography{biblioIR}

\begin{thebibliography}{10}
\expandafter\ifx\csname url\endcsname\relax
  \def\url#1{\texttt{#1}}\fi
\expandafter\ifx\csname urlprefix\endcsname\relax\def\urlprefix{URL }\fi
\expandafter\ifx\csname href\endcsname\relax
  \def\href#1#2{#2} \def\path#1{#1}\fi

\bibitem{perelson1993dynamics}
A.~S. Perelson, D.~E. Kirschner, R.~De~Boer, Dynamics of {HIV} infection of
  {CD4+ T} cells, Mathematical biosciences 114~(1) (1993) 81--125.

\bibitem{legrand2003vivo}
M.~Legrand, E.~Comets, G.~Aymard, R.~Tubiana, C.~Katlama, B.~Diquet, An in vivo
  pharmacokinetic/pharmacodynamic model for antiretroviral combination, HIV
  Clinical trials 4~(3) (2003) 170--183.

\bibitem{perelson2013modeling}
A.~S. Perelson, R.~M. Ribeiro, Modeling the within-host dynamics of {HIV}
  infection, BMC biology 11~(1) (2013) 96.

\bibitem{ciupe2007modeling}
S.~M. Ciupe, R.~M. Ribeiro, P.~W. Nelson, A.~S. Perelson, Modeling the
  mechanisms of acute hepatitis b virus infection, Journal of theoretical
  biology 247~(1) (2007) 23--35.

\bibitem{herrmann2000hepatitis}
E.~Herrmann, A.~U. Neumann, J.~M. Schmidt, S.~Zeuzem, Hepatitis c virus
  kinetics, Antiviral therapy 5~(2) (2000) 85--90.

\bibitem{neumann1998hepatitis}
A.~U. Neumann, N.~P. Lam, H.~Dahari, D.~R. Gretch, T.~E. Wiley, T.~J. Layden,
  A.~S. Perelson, Hepatitis c viral dynamics in vivo and the antiviral efficacy
  of interferon-$\alpha$ therapy, Science 282~(5386) (1998) 103--107.

\bibitem{canini2014viral}
L.~Canini, A.~S. Perelson, Viral kinetic modeling: state of the art, Journal of
  pharmacokinetics and pharmacodynamics 41~(5) (2014) 431--443.

\bibitem{baccam2006kinetics}
P.~Baccam, C.~Beauchemin, C.~A. Macken, F.~G. Hayden, A.~S. Perelson, Kinetics
  of influenza {A} virus infection in humans, Journal of virology 80~(15)
  (2006) 7590--7599.

\bibitem{smith2011influenza}
A.~M. Smith, A.~S. Perelson, Influenza {A} virus infection kinetics:
  quantitative data and models, Wiley Interdisciplinary Reviews: Systems
  Biology and Medicine 3~(4) (2011) 429--445.

\bibitem{hernandez2019passivity}
G.~Hernandez-Mejia, A.~Y. Alanis, M.~Hernandez-Gonzalez, R.~Findeisen, E.~A.
  Hernandez-Vargas, Passivity-based inverse optimal impulsive control for
  {Influenza} treatment in the host, IEEE Transactions on Control Systems
  Technology (2019).

\bibitem{nikin2015role}
R.~Nikin-Beers, S.~M. Ciupe, The role of antibody in enhancing dengue virus
  infection, Mathematical biosciences 263 (2015) 83--92.

\bibitem{nikin2018modelling}
R.~Nikin-Beers, S.~M. Ciupe, Modelling original antigenic sin in dengue viral
  infection, Mathematical medicine and biology: a journal of the IMA 35~(2)
  (2018) 257--272.

\bibitem{Nguyen15}
V.~Nguyen, S.~Binder, A.~Boianelli, M.~Meyer-Hermann, E.~A. Hernandez-Vargas,
  {Ebola Virus Infection Modelling and Identifiability Problems}, Frontiers in
  microbiology 6 (05 2015).
\newblock \href {https://doi.org/10.3389/fmicb.2015.00257}
  {\path{doi:10.3389/fmicb.2015.00257}}.

\bibitem{van2017reproduction}
P.~van~den Driessche, Reproduction numbers of infectious disease models,
  Infectious Disease Modelling 2~(3) (2017) 288--303.

\bibitem{murase2005stability}
A.~Murase, T.~Sasaki, T.~Kajiwara, Stability analysis of pathogen-immune
  interaction dynamics, Journal of Mathematical Biology 51~(3) (2005) 247--267.

\bibitem{smith2003virus}
H.~L. Smith, P.~De~Leenheer, Virus dynamics: a global analysis, SIAM Journal on
  Applied Mathematics 63~(4) (2003) 1313--1327.

\bibitem{abuin2020dynamical}
P.~Abuin, A.~Anderson, A.~Ferramosca, E.~A. Hernandez-Vargas, A.~H. Gonzalez,
  Dynamical characterization of antiviral effects in covid-19, arXiv preprint
  arXiv:2012.15585 (2020).

\bibitem{ciupe2017host}
S.~M. Ciupe, J.~M. Heffernan, In-host modeling, Infectious Disease Modelling
  2~(2) (2017) 188--202.

\bibitem{cao2017mechanisms}
P.~Cao, J.~M. McCaw, The mechanisms for within-host influenza virus control
  affect model-based assessment and prediction of antiviral treatment, Viruses
  9~(8) (2017) 197.

\bibitem{rawlings2017model}
J.~B. Rawlings, D.~Q. Mayne, M.~Diehl, Model predictive control: theory,
  computation, and design, Vol.~2, Nob Hill Publishing Madison, WI, 2017.

\bibitem{blanchini2008set}
F.~Blanchini, S.~Miani, Set-theoretic methods in control, Springer, 2008.

\bibitem{abuin2020char}
P.~Abuin, A.~Anderson, A.~Ferramosca, E.~A. Hernandez-Vargas, A.~H. Gonzalez,
  {Characterization of SARS-CoV-2 Dynamics in the Host}, Annual Reviews in
  Control (2020).

\bibitem{eftimie2016mathematical}
R.~Eftimie, J.~J. Gillard, D.~A. Cantrell, Mathematical models for immunology:
  current state of the art and future research directions, Bulletin of
  mathematical biology 78~(10) (2016) 2091--2134.

\bibitem{hernandez2019modeling}
E.~A. Hernandez-Vargas, {Modeling and Control of Infectious Diseases in the
  Host: With MATLAB and R}, Academic Press, 2019.

\bibitem{dobrovolny2013assessing}
H.~M. Dobrovolny, M.~B. Reddy, M.~A. Kamal, C.~R. Rayner, C.~A. Beauchemin,
  Assessing mathematical models of influenza infections using features of the
  immune response, PloS one 8~(2) (2013) e57088.

\bibitem{vergnaud2005assessing}
J.-M. Vergnaud, I.-D. Rosca, Assessing bioavailablility of drug delivery
  systems: mathematical modeling, CRC press, 2005.

\bibitem{rivadeneira2017control}
P.~S. Rivadeneira, A.~Ferramosca, A.~H. Gonz{\'a}lez, Control strategies for
  nonzero set-point regulation of linear impulsive systems, IEEE Transactions
  on Automatic Control 63~(9) (2018) 2994--3001.

\bibitem{HernandezMejia17}
G.~Hernandez-Mejia, A.~Alanis, E.~A. Hernandez-Vargas, {Inverse Optimal
  Impulsive Control Based Treatment of Influenza Infection}, IFAC-PapersOnLine
  50 (2017) 12185--12190.
\newblock \href {https://doi.org/10.1016/j.ifacol.2017.08.2272}
  {\path{doi:10.1016/j.ifacol.2017.08.2272}}.

\bibitem{rivadeneira2012impulsive}
P.~S. Rivadeneira, C.~H. Moog, Impulsive control of single-input nonlinear
  systems with application to hiv dynamics, Applied Mathematics and Computation
  218~(17) (2012) 8462--8474.

\bibitem{dahari2007modeling}
H.~Dahari, A.~Lo, R.~M. Ribeiro, A.~S. Perelson, Modeling hepatitis {C} virus
  dynamics: liver regeneration and critical drug efficacy, Journal of
  theoretical biology 247~(2) (2007) 371--381.

\bibitem{dobrovolny2011neuraminidase}
H.~M. Dobrovolny, R.~Gieschke, B.~E. Davies, N.~L. Jumbe, C.~A. Beauchemin,
  Neuraminidase inhibitors for treatment of human and avian strain influenza:
  {A} comparative modeling study, Journal of theoretical biology 269~(1) (2011)
  234--244.

\bibitem{wolfel2020virological}
R.~W{\"o}lfel, V.~M. Corman, W.~Guggemos, M.~Seilmaier, S.~Zange, M.~A.
  M{\"u}ller, D.~Niemeyer, T.~C. Jones, P.~Vollmar, C.~Rothe, et~al.,
  Virological assessment of hospitalized patients with {COVID-2019}, Nature
  581~(7809) (2020) 465--469.

\bibitem{vargas2020host}
E.~A. Hernandez-Vargas, J.~X. Velasco-Hernandez, {In-host Modelling of COVID-19
  Kinetics in Humans}, medRxiv (2020).

\bibitem{perko2013differential}
L.~Perko, Differential equations and dynamical systems, Vol.~7, Springer
  Science \& Business Media, 2013.

\bibitem{djorge2020stability}
A.~D'Jorge, A.~L. Anderson, A.~Ferramosca, A.~H. González, M.~Actis, On
  stability of nonzero set-point for non linear impulsive control systems
  (2020).
\newblock \href {http://arxiv.org/abs/2011.12085} {\path{arXiv:2011.12085}}.

\bibitem{khalil2002nonlinear}
H.~K. Khalil, J.~W. Grizzle, Nonlinear systems, Vol.~3, Prentice hall Upper
  Saddle River, NJ, 2002.

\end{thebibliography}

\end{document}